\documentclass{lms}

\usepackage{amssymb}
\usepackage{amsmath}

\newtheorem{theorem}{Theorem}[section] 
\newtheorem{lemma}[theorem]{Lemma}     
\newtheorem{corollary}[theorem]{Corollary}
\newtheorem{proposition}[theorem]{Proposition}

\newnumbered{assertion}{Assertion}    
\newnumbered{conjecture}{Conjecture}  
\newnumbered{definition}[theorem]{Definition}
\newnumbered{hypothesis}{Hypothesis}
\newnumbered{remark}[theorem]{Remark}
\newnumbered{note}{Note}
\newnumbered{observation}{Observation}
\newnumbered{problem}{Problem}
\newnumbered{question}{Question}
\newnumbered{algorithm}{Algorithm}
\newnumbered{example}[theorem]{Example}
\newunnumbered{notation}{Notation} 

\providecommand{\abs}[1]{\lvert#1\rvert}
\providecommand{\norm}[1]{\lVert#1\rVert}

\title[Strassen's and Choquet theorems and transport]
{Applications of Strassen's theorem and Choquet theory to optimal transport problems, to uniformly convex functions and to uniformly smooth functions} 

\author{Krzysztof J. Ciosmak}

\dedication{}

\classno{Primary 46A55, 49N15, 26B25; Secondary 60G42, 90C46, 49N05, 47H05Declaration of interests: none. \\ This research was partly supported by the ERC Starting Grant 802689 CURVATURE}

\begin{document}
\maketitle

\begin{abstract}
We provide a unifying interpretation of various optimal transport problems as a minimisation of a linear functional over the set of all Choquet representations of a given pair of probability measures ordered with respect to a certain convex cone of functions. This allows us to provide novel proofs of duality formulae. Among our tools is Strassen's theorem. 

We provide new formulations of the primal and the dual problem for martingale optimal transport employing a novel representation of the set of extreme points of probability measures in convex order on Euclidean space. We exhibit a link to uniformly convex and uniformly smooth functions and provide a new characterisation of such functions. 

We introduce a notion of martingale triangle inequality. We show that Kantorovich--Rubinstein duality bears an analogy in the martingale setting employing the cost functions that satisfy the inequality. 
\end{abstract}

\section{Introduction}

One of the aims of this article is to introduce a novel, unifying  approach to various optimal transport problems. We refer the reader to \cite{Villani2}, \cite{Villani1} for an extensive account on classical optimal transport and to \cite{Beiglbock1}, \cite{Galichon} for works on martingale optimal transport. The usefulness of the approach is verified by provision of a new, natural formulation of the primal and the dual problems for martingale optimal transport. We also simplify the formulation of the dual problem for martingale optimal transport, under the assumption that the considered cost function satisfies the martingale triangle inequality. The martingale triangle inequality, which we define in this paper, serves us to prove an analogue of the Kantorovich--Rubinstein duality formula in the martingale optimal transport setting.

We investigate martingale optimal transport for cost functions that satisfy the martingale triangle inequality. It turns out that this problem is closely related to classes of uniformly convex and of uniformly smooth functions. We study their continuity properties and provide a new characterisation of such functions.

We compute the set of extreme points of pairs of probabilities in convex order.

The advances obtained in this paper also comprise new proofs of the duality formulae for optimal transport problem, including the Kantorovich--Rubinstein formula and also the duality formula for multi-marginal optimal transport. 

The main tool that we employ to prove the aforementioned results is a new variant of Strassen's disintegration theorem. Originally, see \cite{Strassen}, the theorem was formulated for pairs of measures in convex order. Our result shows that the Strassen's argument may be adapted for pairs of measures that are ordered with respect to a tangent cone to certain convex sets of functions.

Another tool is Choquet theory, see e.g. \cite{Alfsen}, \cite{Phelps}, which allows us to infer the duality formulae from representations of appropriate sets of extreme points.

\subsection{General description of the method}\label{s:general}

The dual problems that arise in various optimal transport problems usually concern maximisation of a linear functional over some convex set of functions. Thus, provided that a maximiser exists, the measures under consideration are naturally ordered with respect to a tangent cone to this convex set at a maximiser.

Suppose we consider a variational problem of maximising 
\begin{equation}\label{eqn:problem}
\int_{\Omega} g \,d\nu-\int_{\Omega} g\,d\mu
\end{equation}
over the functions $g\in \mathcal{K}$, where $\mathcal{K}$ is a convex set of continuous, bounded functions on a locally compact set $\Omega$, stable under taking maxima, containing constants and such that $t+\mathcal{K}\subset\mathcal{K}$ for any constant function $t$.

Suppose that a maximiser $f\in\mathcal{K}$ exists. It follows then that measures $\mu,\nu$ are ordered with respect to the tangent cone $\mathcal{F}$ of $\mathcal{K}$ at $f$, i.e.,  
\begin{equation}\label{eqn:polar}
\int_X \lambda (f-g)\,d\nu\geq \int_X \lambda(f-g)\,d\mu\text{ for all }\lambda\geq 0\text{ and all }g\in\mathcal{K}.
\end{equation}
Strassen's theorem, Theorem \ref{thm:varstrassen}, allows us to infer that the extreme points of measures that satisfy (\ref{eqn:polar}) are necessarily contained in the set of measures of the form $(\delta_x,\eta)$, for some Borel probability measure $\eta$ such that $\delta_x$ is majorised by $\eta$ in the order induced by $\mathcal{F}$.

We consider the problem (\ref{eqn:polar}) in the setting of classical, multi-marginal and martingale optimal transport. To precisely characterise the extreme points of pairs of measures that are ordered with respect to appropriate cone of functions we employ the $c$-convexification method and the symmetry of the problem in the case of multi-marginal optimal transport.

It turns out that for the pairs of measures that are extreme points the duality is trivial, thanks to a construction based on the $c$-convexification. With the help of Choquet's theorem, this allows us to infer the duality for various optimal transport problems. That is, we prove that the maximal value of (\ref{eqn:problem}) is equal to the optimal cost of corresponding transport problem.

%
%

We provide a more detailed description of the results in the following sections.

\subsection{Strassen's theorem}\label{ss:strassen}

In the martingale optimal transport problem, one considers two Borel probability measures $\mu,\nu$ on $\mathbb{R}^n$ with finite first moments that are in convex order. That is
\begin{equation}\label{eqn:order}
\int_{\mathbb{R}^n}f\,d\mu\leq \int_{\mathbb{R}^n}f\,d\nu\text{ for any convex }f\colon\mathbb{R}^n\to\mathbb{R}.
\end{equation}
A theorem of Strassen, see \cite{Strassen}, implies that there exists a Markov kernel $P$ from $\mathbb{R}^n$ to $\mathbb{R}^n$ such that for any $x\in\mathbb{R}^n$ the pair $(\delta_x,P(\cdot,x))$ is in convex order and $\nu$ can be represented as an integral
\begin{equation*}
\nu=\int_{\mathbb{R}^n}P(\cdot,x)\,d\mu(x).
\end{equation*}
%

Over the time several variants of the theorem have been proven. Already the paper \cite{Valadier} published in 1974 observes that Strassen's theorem has sparked a substantial research interest. 
Since then even more results have been proven in this direction. In the paper \cite{Meyer-Nieberg} several Strassen's type disintegration theorems are surveyed. In particular, a version for Dedekind complete Riesz spaces, which implies the other versions presented in that paper.
Let us also mention a version for convex cones with localized order structure, see \cite{Fuchssteiner}, and a non-commutative version, see \cite{Hackenbroch}.

To provide some intuition on the theorem, let us recall a standard fact which says that if a linear functional is bounded by a finite sum of semi-norms, then it may be decomposed into a sum of linear functionals, each of which is bounded by one of the considered semi-norms. This fact is a cornerstone of Strassen's theorem and its generalisations. For more in-depth discussion see \cite{Neumann}.

Consider a cone of continuous and bounded functions $\mathcal{F}$ on a locally compact Polish space $\Omega$ that is stable under taking maxima and contains constant functions. Strassen's theorem may be generalised to pairs $(\mu,\nu)$ of probability measures which are ordered with respect to $\mathcal{F}$, which amounts to a demand that 
\begin{equation}\label{eqn:balayage}
\int_{\Omega}f\,d\mu\leq \int_{\Omega}f\,d\nu\text{ for any }f\in\mathcal{F}.
\end{equation}
In this paper, we employ this result, see Theorem \ref{thm:extr} and Proposition \ref{pro:extr}, so that in place of a cone $\mathcal{F}$ one may put a tangent cone to a convex set of continuous and bounded functions that is stable under maxima, contains constants and is closed under addition of constants. This is precisely the setting we already discussed in Section \ref{s:general}, cf. (\ref{eqn:polar}).

Let us mention here that pairs $(\mu,\nu)$ that satisfy (\ref{eqn:balayage}) are said to be in $\mathcal{F}$-balayage. The balayage theory has been developed by Choquet, Mokobodzki and Sibony, see \cite{Choquet}, \cite{Mokobodzki1} and \cite{Mokobodzki2}. Let us also mention the book of Meyer \cite[Chapter 11]{Meyer}.
For a reference on the related topic of convex cones we refer the reader to the books \cite{Fuchssteiner1} and \cite{Becker}.
The relation of the balayage theory to the optimal transport problems has been already established in \cite{Bowles}. We refer also to a paper of Ghoussoub \cite{Ghoussoub2} for further developments of this line of research.
The theory is a sub-field of the potential theory, which has already found many applications to martingale optimal transport.

\subsection{Martingale optimal transport}

Recently, great attention has been paid to the problem of martingale optimal transport, especially in the multi-dimensional setting. The initial interest in this problem stems from its applications to mathematical finance \cite{Beiglbock1}, \cite{Galichon}, and its link to the Skorokhod embedding problem  \cite{Beiglbock}, \cite{Obloj}. Let us also mention contributions of Ghoussoub, Kim, Lim  \cite{Ghoussoub}, by De March, Touzi \cite{DeMarch3}, \cite{DeMarch2}, \cite{DeMarch1} and by Ob\l\'oj and Siorpaes \cite{Siorpaes}, which deal with the important notion of irreducible convex paving.

Suppose that $\mu,\nu$ is a pair of Borel probability measures in convex order, see Section \ref{ss:strassen}, with finite first moments.
The theorem of Strassen implies that there exists a coupling $\pi$ on $\mathbb{R}^n\times\mathbb{R}^n$ with respective marginals $\mu$ and $\nu$, such that if $(X,Y)$ is a random vector distributed according to $\pi$ then the pair $(X,Y)$ is a one-step martingale, i.e., $\mathbb{E}(Y|X)=X$. The martingale optimal transport problem is to find such coupling $\pi$ that minimises the integral
\begin{equation*}
\int_{\mathbb{R}^n\times\mathbb{R}^n}c\,d\pi
\end{equation*}
for a given measurable cost function $c\colon\mathbb{R}^n\times\mathbb{R}^n\to\mathbb{R}$. The dual problem is to find the maximal value of 
\begin{equation*}
\int_{\mathbb{R}^n}u \,d\mu-\int_{\mathbb{R}^n}v\,d\nu
\end{equation*}
among all pairs $u,v\colon\mathbb{R}^n\to\mathbb{R}$ of continuous functions such that there exists $\gamma\colon\mathbb{R}^n\to\mathbb{R}^n$ satisfying
\begin{equation}\label{eqn:ori}
u(x)-v(y)+\langle \gamma(x),y-x\rangle\leq c(x,y)\text{ for all }x,y\in\mathbb{R}^n.
\end{equation}

In Theorem \ref{thm:ex} we study the set of extreme points of pairs of probability measures in convex order. It consists of pairs of the form
\begin{equation*}
\Big(\delta_x,\sum_{i=1}^{d+1}t_i\delta_{x_i}\Big)\text{ with }x=\sum_{i=1}^{d+1}t_ix_i
\end{equation*}
for some positive $t_1,\dotsc,t_{d+1}$ summing up to one, some affinely independent points $x_1,\dotsc,x_{d+1}\in\mathbb{R}^n$ and $d\leq n$. Using this assertion, Theorem \ref{thm:extr} and Proposition \ref{pro:extr}, we provide a novel formulation of the primal problem and the dual problem in martingale optimal transport, see Theorem \ref{thm:dualmart}. Namely, instead of looking at the infimum among all martingale couplings, we look at the infimum among all Choquet representations of a given pair of probability measures by means of extreme points of measures in convex order. The advantage of this viewpoint is that it is a natural counterpart to a redefined dual problem, which relies on an intrinsic definition of the considered class of functions.

Employing Theorem \ref{thm:ex}, we prove that the set of pairs of continuous functions satisfying (\ref{eqn:ori}) is equal to the set of all pairs $u,v\colon\mathbb{R}^n\to\mathbb{R}$ for which
\begin{equation}\label{eqn:orip}
u\Big(\sum_{i=1}^{n+1}t_ix_i\Big)-\sum_{i=1}^{n+1}t_i v(x_i)\leq \sum_{j=1}^{n+1}t_jc\Big(\sum_{i=1}^{n+1}t_ix_i,x_j\Big)
\end{equation}
 for all $x_1,\dotsc,x_{n+1}\in\mathbb{R}^n$ and all non-negative $t_1,\dotsc,t_{n+1}$ that sum up to one. 
The proof employs the Hahn--Banach theorem and the Riesz' representation theorem.
It works also in the case of functions on general convex sets $K\subset\mathbb{R}^n$. Moreover, we can prove equivalence of conditions (\ref{eqn:ori}) and (\ref{eqn:orip}) for each barycentre $x=\sum_{i=1}^{n+1}\lambda_ix_i$ separately.
For the precise formulation of the result we refer to Theorem \ref{thm:col}. 

\subsection{Martingale triangle inequality}

One of new results in this paper is an analogue in the martingale optimal transport setting of Kantorovich--Rubinstein duality. In the classical optimal transport problem with metric cost function Kantorovich--Rubinstein duality tells that the optimal cost of transporting $\mu$ to $\nu$ is equal to the supremum of integrals of $1$-Lipschitz functions with respect to the signed measure $\mu-\nu$. In other words, if the cost function is given by a metric, one may restrict the set of pairs of functions over which the supremum is taken to the set of pairs of equal functions.

Kantorovich--Rubinstein duality has found numerous applications. One of them is a localisation technique, which allows to reduce certain high-dimensional problems to a collection of one-dimensional problems. In \cite{Klartag} the technique is studied in the context of weighted Riemannian manifolds, while in \cite{Cavalletti} it is adapted to metric measure spaces satisfying curvature-dimension condition. Certain higher-dimensional generalisations of the technique, proposed in \cite{Klartag}, are studied in \cite{Ciosmak2} and in \cite{Ciosmak3}.

Similar simplification can be accomplished in the martingale optimal transport problem provided that a cost functions satisfies the \emph{martingale triangle inequality}, which we define in this paper, see Definition \ref{defin:martingale}.
Let  $c\colon\mathbb{R}^n\times\mathbb{R}^n\to\mathbb{R}$. The inequality is satisfied by $c$ if 
\begin{equation*}
\sum_{i=1}^{n+1}\lambda_i c(x,x_i)-c\Big(x,\sum_{i=1}^{n+1}\lambda_ix_i\Big)\leq \sum_{i=1}^{n+1}\lambda_i c\Big(\sum_{j=1}^{n+1}\lambda_jx_j,x_i\Big)
\end{equation*} 
for any $x,x_1,\dotsc,x_{n+1}\in\mathbb{R}^n$ and any non-negative $t_1,\dotsc,t_{n+1}$ that sum up to one.

An equivalent way to state the inequality is to require that for any $x\in\mathbb{R}^n$ and any bounded one-step martingale $(X_0,X_1)$ it holds
\begin{equation*}
\mathbb{E} c(x,X_1)-\mathbb{E} c(x,X_0)\leq\mathbb{E}c(X_0,X_1).
\end{equation*}
Another characterisation may be provided in the spirit of (\ref{eqn:ori}).

We prove that if $c$ satisfies the inequality and vanishes on the diagonal, then the value of the dual problem will not be changed if we restrict ourselves to pairs of equal functions, see Theorem \ref{thm:clc}. That is, it is enough to consider functions $f$ such that 
\begin{equation}\label{eqn:one}
f\Big(\sum_{i=1}^{n+1}t_ix_i\Big)-\sum_{i=1}^{n+1}t_i f(x_i)\leq \sum_{j=1}^{n+1}t_jc\Big(\sum_{i=1}^{n+1}t_ix_i,x_j\Big),
\end{equation}
for all $x_1,\dotsc,x_{n+1}\in\mathbb{R}^n$ and all non-negative $t_1,\dotsc,t_{n+1}$ that sum up to one.

This class of functions is also investigated in Theorem \ref{thm:onefunction}.

\subsection{Continuity properties}

In Theorem \ref{thm:lipm} we study continuity properties of the class of functions considered in Theorem \ref{thm:clc}. We prove that, under certain mild continuity assumptions on the function $c$, any function that satisfies (\ref{eqn:one}) is locally Lipschitz. The argument adapts the reasoning of \cite{Wayne}. Here we do not need to assume that $c$ satisfies the martingale triangle inequality.

This result complements the standard knowledge about convex functions, cf. \cite{Wayne}.

In Theorem \ref{thm:onefunction} we show that the characterisation of convex functions by means of subdifferentials bears an analogy in the setting of functions satisfying (\ref{eqn:one}). 

These results prove to be useful also for characterisation of uniformly smooth and uniformly convex functions.

\subsection{Uniformly convex and uniformly smooth functions}

As an application of the developed approach, we provide a characterisation of uniformly convex and uniformly smooth functions that complements results of Az\`e and Penot \cite{Aze} and of Z\u{a}linescu \cite{Zalinescu}, see Theorem \ref{thm:uni}. We refer the reader to \cite{Aze} for a detailed study of various characterisation of these functions. 

These and related classes of functions have been extensively studied and found numerous applications. Let us mention works of Asplund \cite{Asplund}, Asplund and Rockafellar \cite{Asplund2}, a related work of \v{S}mulyan \cite{Smulyan}. Applications include classical gradient methods \cite{Dunn}, \cite{Lucchetti} and  proximal algorithm \cite{Rockafellar3}.

Let $\sigma\colon \mathbb{R}\to\mathbb{R}$. Let us recall that a function $f\colon\mathbb{R}^n\to\mathbb{R}$ is $\sigma$-uniformly convex provided that there exists $\gamma\colon\mathbb{R}^n\to\mathbb{R}$ such that
\begin{equation*}
f(x)+\sigma(\norm{y-x})+\langle \gamma(x),y-x\rangle\leq f(y)\text{ for all } x,y\in\mathbb{R}^n, \lambda\in [0,1].
\end{equation*}
If $\sigma$ is locally Lipschitz with $\sigma(0)=0$, we prove that this is equivalent to 
\begin{equation*}
f\Big(\sum_{i=1}^{n+1}\lambda_ix_i\Big)-\sum_{i=1}^{n+1}\lambda_i f(x_i)\leq -\sum_{i=1}^{n+1}\lambda_i \sigma \Big(\Big\lVert\sum_{j=1}^{n+1}\lambda_jx_j-x_i\Big\rVert\Big)
\end{equation*}
for all $x_1,\dotsc,x_{n+1}\in \mathbb{R}^n$ and all non-negative $\lambda_1,\dotsc,\lambda_{n+1}$ that sum up to one. Another equivalent condition is that for any bounded one-step martingale $(X_0,X_1)$ there is
\begin{equation*}
\mathbb{E} \big(f(X_0)-f(X_1)\big)\leq -\mathbb{E}\sigma (\norm{X_0-X_1}).
\end{equation*}
A similar result holds true for $\sigma$-uniformly smooth functions, that is functions $g\colon\mathbb{R}^n\to\mathbb{R}$ such that there exists $\gamma\colon\mathbb{R}^n\to\mathbb{R}$ such that 
\begin{equation*}
g(x)+\sigma(\norm{y-x})+\langle \gamma(x),y-x\rangle\geq g(y)\text{ for all }x,y\in\mathbb{R}^n, \lambda\in [0,1].
\end{equation*}
We refer the reader to Theorem \ref{thm:uni} for a general statement of the equivalences, which hold true also for functions on an arbitrary open, convex set $K\subset\mathbb{R}^n$.

\subsection{Optimal transport}

For the sake of illustration of the method and for completeness, we also provide proofs of Kantorovich duality, Kantorovich--Rubinstein duality and duality in the multi-marginal optimal transport problem.

Let us recall the topic of classical optimal transport and briefly review the literature.

Suppose we are given two Borel probability measures $\mu,\nu$ on topological spaces $X,Y$ and a measurable cost function $c\colon  X \times Y \to \mathbb{R}$. The optimal transport problem, proposed by Monge \cite{Monge}, is concerned with finding a measurable map $T\colon X \to Y $ such that it pushes $\mu$ forward to $\nu$, $T_{\#}\mu=\nu$, i.e., for any Borel set $A\subset Y$ there is $\nu(A)=\mu(T^{-1}(A))$, and such that the integral
\begin{equation*}
\int_X c(x,T(x))\,d\mu(x)
\end{equation*}
is minimal. In \cite{Kantorovich2}, \cite{Kantorovich} Kantorovich proposed a relaxed version of the problem. Namely, instead of looking for an optimal map, one seeks for a coupling $\pi$, i.e., a Borel probability measure on $X\times Y$, whose marginal distributions are $\mu$ and $\nu$, that minimises the integral
\begin{equation}\label{eqn:minik}
\int_{X\times Y}c \,d\pi.
\end{equation}
Kantorovich also provided a dual formulation of the problem. In the dual problem one wants to maximise 
\begin{equation}\label{eqn:maxik}
\int_X u\,d\mu+\int_Y v\,d\nu.
\end{equation}
among all pairs of continuous functions $u\colon X\to \mathbb{R}$ and $v\colon Y\to\mathbb{R}$ such that 
\begin{equation}\label{eqn:condid}
u(x) + v(y) \leq c(x, y),
\end{equation}
for all $x\in X$ and $y\in Y$.

It has been proven, see e.g. \cite[Theorem 1.3]{Villani2}, that the minimal value of (\ref{eqn:minik}) and the maximal value of (\ref{eqn:maxik}) coincide, under the assumption that $c$ is lower semi-continuous and the considered spaces are Polish. We reprove this result using tools of Choquet theory and Strassen's theorem, Theorem \ref{thm:extr} together with Proposition \ref{pro:extr}. For the formulation of the duality results and their proofs see Theorem \ref{thm:opti} for the classical problem and Theorem \ref{thm:kr} for the optimal transport problem with metric cost function.

Let us note that this result may be also obtained under the assumption that $c$ is merely Borel measurable and bounded from above, see \cite{Kellerer2}, via Choquet's capacitability theorem.

Let us also mention a work of Ramachandran and R\"uschendorf \cite{Ramachandran}, \cite{Ramachandran2}, where the problem is studied beyond the setting of Polish space. The assumption here, which is also shown to be essential, is that one of the coordinate spaces is perfect.

The vast literature on this problem includes \cite{Beiglbock5}, where the authors provide a suitable relaxation of the optimal transport problem: the cost is defined as a limit of partial optimal transport costs.  In there, it is proven that, under this relaxation, the duality theorem holds true for Borel cost functions.

A long line of research on related questions includes also works of Dudley \cite{Dudley1}, \cite{Dudley2}, Fernique \cite{Fernique}, Mikami \cite{Mikami}, de Acosta \cite{deAcosta}, Mikami and Thieullen \cite{Mikami2}, Beiglb\"ock and  Schachermayer \cite{Beiglbock6}. We refer the reader to the bibliographical notes in a book of Villani \cite[pp. 97-104]{Villani1} for further details about the history of the problem and its importance for the study of geometry of optimal transport plans.

\subsection{Multi-marginal optimal transport}

Let us also mention the problem of multi-marginal optimal transport. Suppose we are given Borel probability measures $\mu_1,\dotsc,\mu_k$ on respective topological spaces $X_1,\dotsc,X_k$ and a Borel measurable function $c\colon X_1\times\dots\times X_k\to\mathbb{R}$. Among, all probability measures $\pi$ on the product $X_1\times\dots \times X_k$ with respective marginals $\mu_1,\dotsc,\mu_k$, we look for the one that minimises the integral
\begin{equation*}
\int_{X_1\times\dots \times X_k}c\,d\pi.
\end{equation*}
We refer the reader to \cite{Pass1} for an introduction to the topic with several applications provided, to \cite{Pass2} and to \cite{Pass}. Let us mention that already in \cite{Kellerer2}  a duality formula for multiple-marginal transport problems is provided. 

Our method can be applied to prove a duality formula also in this setting, see Theorem \ref{thm:optimany}. Here, similarly as in the two-marginal case, we employ a version of $c$-convexification method, see Lemma \ref{lem:convexification}. The method has been already exploited in the context of multi-marginal optimal transport in \cite{Carlier} and in \cite{Swiech}. Yet,  our result is more general, as it combines usual $c$-convexification with an extension-type result. Also, we prove that the supremum considered in the dual problem may be taken over a set consisting of bounded functions, see Lemma \ref{lem:boundedness}.

\subsection{Further research}

One of the possible future research directions related to the findings of this paper lie in generalisations of the aforementioned localisation technique. Let us note that the irreducible convex paving, investigated in the setting of martingale transport, can be viewed as a variant of the technique. Let us note here that while the condition of convex ordering is preserved by the paving, it has not yet been proven that curvature-dimension condition would localise as well. However, in \cite{Caravenna2}, it has been proven that the disintegration with respect to a partitioning into faces of a convex function yields measures absolutely continuous with respect to Hausdorff measures of appropriate dimensions, provided that the initial measure was also absolutely continuous. We see this result as a first step in establishing localisation of curvature-dimension condition in the setting of martingale transport. 

Since Kantorovich--Rubinstein duality plays a vital r\^ole for localisation, we think that the duality result established for the cost functions that satisfy martingale triangle inequality, and optimal transport with respect to such cost functions, are of substantial interest.

Another direction where the current developments may be useful is concerned with localisation of probability measures that are ordered with respect to a cone of functions.

\subsection*{Acknowledgements}

The financial support of Clardendon Fund, St John’s College in Oxford and EPSRC is gratefully acknowledged. The author would like to acknowledge the kind hospitality of the Erwin Sch\"odinger International Institute for Mathematics and Physics where parts of this research were developed under the frame of the Thematic Programme on Optimal Transport. This research was also partly supported by the ERC Starting Grant 802689 CURVATURE.
 
\subsection*{Outline of the article}

In Section \ref{s:strassen} we recall necessary definitions and prove the variant of Strassen's disintegration theorem; see Theorem \ref{thm:varstrassen}, Theorem \ref{thm:extr} and Proposition \ref{pro:extr}.

In Section \ref{s:martingale} we characterise the extreme points of pairs of Borel probability measures in convex order; see Theorem \ref{thm:ex}. We prove a duality result for martingale optimal transport provided that there exists a maximiser of the dual problem; see Theorem \ref{thm:dualmart}.

In Section \ref{s:duality} we investigate class of functions that appear in the dual problem to the martingale optimal transport; see Theorem \ref{thm:col}.

In Section \ref{s:triangle} we introduce the notion of the martingale triangle inequality and prove that if the inequality is satisfied by a cost function, then in the dual problem to martingale optimal transport the consideration may be restricted to pairs of equal functions; see Definition \ref{defin:martingale} and Theorem \ref{thm:clc}.

In Section \ref{s:continuity} we provide a characterisation and study continuity properties of the class of functions considered in Section \ref{s:triangle}; see Theorem \ref{thm:lipm} and Theorem \ref{thm:onefunction}.

In Section \ref{s:uniform} we apply the results of Section \ref{s:continuity} and obtain a novel characterisation of uniformly convex and uniformly smooth functions; see Theorem \ref{thm:uni}.

In Section \ref{s:transport} we provide a proof of Kantorovich duality in the two-marginal case; see Theorem \ref{thm:opti}.

In Section \ref{s:Kant-Rub} we prove Kantorovich--Rubinstein duality, i.e., the duality result for a cost function given by a metric; see Theorem \ref{thm:kr}.

In Section \ref{s:multi} we provide a proof of Kantorovich duality in the multi-marginal setting; see Theorem \ref{thm:optimany}.

\section{Variant of Strassen's theorem}\label{s:strassen}

We begin with the following version of Strassen's theorem \cite[Theorem T51, p. 244]{Meyer}.  We refer the reader also to \cite[Section 19.8]{Aliprantis} and to \cite{Strassen}.

\begin{theorem}\label{thm:varstrassen}
Let $X$ be a separable Banach space, let $(\Omega,\Sigma,\mu)$ be a probability space. Let $\omega\mapsto h_{\omega}$ be a map from $\Omega$ to continuous, sublinear and positively homogeneous functions on $X$, which is weakly measurable, that is, for every $x\in X$ the map $\omega\mapsto h_{\omega}(x)$ is $\Sigma$-measurable, and such that there exists non-negative number $c$ such that
\begin{equation}\label{eqn:ccc}
\abs{h_{\omega}(x)}\leq c\norm{x}\text{ for all }x\in X.
\end{equation}
Set 
\begin{equation*}
h(x)=\int_{\Omega}h_{\omega}(x)\,d\mu(\omega).
\end{equation*}
For a functional $x^*\in X^*$ the following conditions are equivalent:
\begin{enumerate}[i)]
\item\label{i:first} $x^*\leq h$,
\item\label{i:second} there exists a map $\omega\mapsto x_{\omega}^*$ from $\Omega$ to $X^*$ which is weakly measurable, in the sense that $\omega\mapsto x_{\omega}^*(x)$ is measurable for any $x\in X$, and such that $x_{\omega}^*\leq h_{\omega}$ for $\mu$-almost every $\omega\in\Omega$ and for all $x\in X$
\begin{equation*}
x^*(x)=\int_{\Omega}x_{\omega}^*(x)\,d\mu(\omega).
\end{equation*}
\end{enumerate}
\end{theorem}

Before we proceed with applications of Strassen's theorem, let us recall definitions.

\begin{definition}
If $(\Omega,\Sigma)$ and $(\Xi, \Theta)$ are measurable spaces, then a Markov kernel $P$ from $\Omega$ to $\Xi$ is a real function on $\Theta\times \Omega$ such that for any point $\omega\in\Omega$, $P(\cdot,\omega)$ is a probability measure on $\Theta$ and for any $A\in \Theta$, $P(A,\cdot)$ is $\Sigma$-measurable. 

If $\mu$ is a probability measure on $\Sigma$, then we define $P\mu$ to be a probability measure on $\Theta$ such that
\begin{equation*}
P\mu(A)=\int_{\Omega} P(A,\omega)\,d\mu(\omega)\text{ for all }A\in\Theta.
\end{equation*}
\end{definition}

We shall denote by $\mathcal{C}(\Omega)$ the Banach space of bounded continuous functions on a topological space $\Omega$ and by $\mathcal{M}(\Omega)$ the Banach space of signed Borel measures on $\Omega$ normed by total variation. By $\mathcal{P}(\Omega)$ we shall denote the set of Borel probability measures on $\Omega$. A subset $\mathcal{K}$ of $\mathcal{C}(\Omega)$ is said to be stable under maxima provided that the maximum $f\vee g$ of two functions   $f,g\in\mathcal{K}$ belongs to $\mathcal{K}$.

Below we state Strassen's theorem for \emph{balayage cones}, i.e., convex cones that are stable under maxima and contain constants (cf. \cite{Ghoussoub2}). Typically the theorem is stated for compact spaces -- cf. \cite[Section 19.8]{Aliprantis} and \cite[T53, p. 246]{Meyer} -- while we provide a version for locally compact Polish spaces.

\begin{theorem}\label{thm:extr}
Let $\Omega$ be a locally compact Polish space. Let $\mathcal{F}$ be a convex cone in $\mathcal{C}(\Omega)$ that is stable under maxima and contains constants.
Suppose that $\mu,\nu$ are Borel probability measures such that 
\begin{equation}\label{eqn:maj}
\int_{\Omega}f\,d\mu\leq\int_{\Omega} f\,d\nu
\end{equation}
for all $f\in\mathcal{F}$.
Then there exists a Markov kernel $P$ form $\Omega$ to $\Omega$ such that  $\nu=P\mu$ and such that for every $\omega\in\Omega$
\begin{equation*}
\int_{\Omega}f\,d\delta_{\omega}\leq\int_{\Omega} f\,dP(\cdot,\omega)
\end{equation*}
for all $f\in\mathcal{F}$.
Moreover, the set of extreme points of the set of pairs of Borel probability measures $(\mu,\nu)$ that satisfy (\ref{eqn:maj}) is contained in the set of pairs of the form $(\delta_{\omega},\eta)$ for some $\omega\in \Omega$ and some Borel probability measure $\eta$ on $\Omega$.
\end{theorem}
\begin{proof}
Set $X=\mathcal{C}(\Omega)$ to be the Banach space of all continuous bounded functions on $\Omega$. Let $x^*$ be an element of $X^*$ represented by a measure $\nu\in\mathcal{M}(\Omega)$. Set for $\omega\in\Omega$
\begin{equation*}
h_{\omega}(x)=\inf\big\{-y(\omega)\mid y\in\mathcal{F}, -y\geq x\big\}.
\end{equation*}
Then, as $\mathcal{F}$ is a convex cone, $h_{\omega}$ is sublinear and positively homogeneous. It is moreover continuous.
Indeed, 
\begin{equation*}
\abs{h_{\omega}(x)}\leq\norm{x},\text{ for all }\omega\in \Omega \text{ and }x\in X.
\end{equation*}
The lower bound $h_{\omega}(x)\geq -\norm{x}$ follows directly from the definition, while the upper bound $h_{\omega}(x)\leq \norm{x}$ follows as $\mathcal{F}$ contains constants and is closed under maxima.
%
As $\mathcal{C}(\Omega)$ is separable, so is its subset 
\begin{equation*}
\big\{-y\mid y\in\mathcal{K}, -y\geq x\big\}.
\end{equation*}
By the assumption that $\mathcal{F}$ is stable under maxima, $\omega\mapsto h_{\omega}(x)$ is a pointwise limit of a sequence $(-y_k)_{k=1}^{\infty}$ with $y_k\in\mathcal{F}$. We may moreover assume that 
\begin{equation*}
\abs{y_k(\omega)}\leq \norm{x} \text{ for }\omega\in \Omega.
\end{equation*}
By the assumption on $\mu,\nu$
\begin{equation*}
x^*(x)=\int_{\Omega} x\,d\nu \leq \int_{\Omega} (-y_k)(\omega)\,d\nu(\omega)\leq\int_{\Omega}(-y_k)(\omega)\,d\mu(\omega).
\end{equation*}
Now, by the dominated convergence theorem it follows that
\begin{equation*}
x^*(x)\leq \int_{\Omega}h_{\omega}(x)\,d\mu(\omega).
\end{equation*} 
Observe that the Banach space $\mathcal{C}_0(\Omega)$ is separable.  By Theorem \ref{thm:varstrassen} 
we see that there is a weakly measurable function $\omega\mapsto x_{\omega}^*$ with values in $X^*$ such that 
\begin{equation}\label{eqn:rep}
x^*(x)=\int_{\Omega}x_{\omega}^*(x) \,d\mu(\omega)\text{ for all }x\in \mathcal{C}(\Omega)
\end{equation}
and $x_{\omega}'\leq h_{\omega}$ for $\mu$-almost every $\omega\in\Omega$.
Here $x_{\omega}'$ is the restriction of $x_{\omega}^*$ to $\mathcal{C}_0(\Omega)$.

Now $\abs{h_{\omega}(x)}\leq\norm{x}$ implies that
for all $x\in\mathcal{C}_0(\Omega)$
\begin{equation*}
x_{\omega}'(x)\leq \norm{x}.
\end{equation*}
By Riesz' representation theorem, there exists a measure $P(\cdot,\omega)$ on $\Omega$ such that for all $h\in \mathcal{C}_0(\Omega)$ there is 
\begin{equation*}
x_{\omega}'(h)=\int_{\Omega}h\, dP(\cdot,\omega).
\end{equation*}
Choose any $h\in \mathcal{C}(\Omega)$. By Ulam's lemma and Urysohn's lemma there exists a bounded, monotone, sequence of non-negative continuous and compactly supported functions $(\phi_n)_{n\in \mathbb{N}}$ that converges pointwise to constant function $1$. In view of (\ref{eqn:rep}),
\begin{equation*}
\int_{\Omega}h\phi_n\,d\nu=\int_{\Omega}\int_{\Omega} h\phi_n \,dP(\cdot,\omega) \,d\mu(\omega).
\end{equation*}
By the dominated convergence theorem there is 
\begin{equation}\label{eqn:repp}
\int_{\Omega}h\,d\nu=\int_{\Omega}\int_{\Omega} h \,dP(\cdot,\omega)\, d\mu(\omega).
\end{equation}
It follows that $P(\cdot,\omega)$ is a probability measure for $\mu$-almost every $\omega$. 
Since $x_{\omega}'\leq h_{\omega}$ for $\mu$-almost every $\omega\in\Omega$, we see that for all $h\in\mathcal{C}_0(\Omega)$ there is
\begin{equation*}
\int_{\Omega} h\,dP(\cdot,\omega)\leq h_{\omega}(h)\text{ for }\mu\text{-almost every }\omega\in\Omega.
\end{equation*}
As $h_{\omega}$ is monotone on $\mathcal{C}(\Omega)$, we see that the above inequality holds true for $h\in\mathcal{C}(\Omega)$ as well.

Observe that if $g\in -\mathcal{F}$, then by the definition of $h_{\omega}$,
\begin{equation*}
h_{\omega}(g)=g\text{, hence }
\int_{\Omega} h\,dP(\cdot,\omega)\leq  g(\omega).
\end{equation*}
It follows that for all $f\in\mathcal{F}$ there is
\begin{equation}\label{eqn:ex}
\int_{\Omega} f\,d\delta_{\omega}\leq \int_{\Omega} f\,dP(\cdot,\omega).
\end{equation}
%
We see that $P$ defines a Markov kernel from $\Omega$ to $\Omega$. 
By (\ref{eqn:repp}), $\nu=P\mu$ and by (\ref{eqn:ex}) we have
\begin{equation}\label{eqn:modify}
\int_{\Omega} f \,d\delta_{\omega}\leq \int_{\Omega} f \,P(\cdot,\omega)
\end{equation}
for $\mu$-almost every $\omega\in\Omega$ and all $f\in\mathcal{F}$. To obtain the desired Markov kernel we modify $P$ on a measurable set of $\mu$-measure zero of $\omega\in \Omega$ such that the inequality (\ref{eqn:modify}) is not valid for some $f\in\mathcal{F}$ by putting $P(\cdot,\omega)=\delta_{\omega}$.  

We have
\begin{equation*}
(\mu,\nu)=\int_{\Omega} (\delta_{\omega},P(\cdot,\omega))\,d\mu(\omega),
\end{equation*}
so, by (\ref{eqn:ex}), the claim about the extreme points follows.
\end{proof}

\begin{proposition}\label{pro:extr}
Suppose that $\mathcal{K}$ is a convex set that is stable under maxima, contains constants and for any constant $c$ it holds $\mathcal{K}+c\subset\mathcal{K}$. 
Let $f\in\mathcal{K}$. Then the tangent cone to $\mathcal{K}$ at $f$, i.e., the cone
\begin{equation*}
\mathcal{F}=\Big\{\lambda(f-g)\mid g\in\mathcal{K},\lambda\geq 0\Big\}
\end{equation*}
is a convex cone, stable under maxima and contains constants.
\end{proposition}
\begin{proof}
Let $\lambda_1,\lambda_2\geq 0$, $g_1,g_2\in\mathcal{K}$. Set $\lambda=\lambda_1\vee\lambda_2$ and 
\begin{equation*}
g=\Bigg(\frac{\lambda_1}{\lambda} g_1+ \Big(1-\frac{\lambda_1}{\lambda}\Big)f \Bigg)\vee\Bigg(\frac{\lambda_2}{\lambda} g_2+ \Big(1-\frac{\lambda_2}{\lambda}\Big)f \Bigg).
\end{equation*}
Then $g\in\mathcal{K}$, thanks to convexity and stability with respect to maxima of $\mathcal{K}$.
Moreover
\begin{equation}
\Big(\lambda_1(f-g_1)\Big)\vee\Big(\lambda_2(f-g_2)\Big)=\lambda(f-g).
\end{equation}
This shows that $\mathcal{F}$ is stable under maxima. The other two claimed properties of $\mathcal{F}$ are trivial to verify.
\end{proof}

\begin{remark}\label{rem:cone}
If we assume moreover that $\mathcal{K}$ contains a cone $\mathcal{F}$, then condition (\ref{eqn:maj}) implies that for all $g\in\mathcal{F}$ one has 
\begin{equation*}
\int_{\Omega}g \,d(\mu-\nu)\leq 0.
\end{equation*}
This is to say, $\mu$ and $\nu$ are in $\mathcal{F}$-balayage.
\end{remark}

The next corollary extends the above result to the case of a pair $(\mu,\nu)$ of measures on two, possibly distinct, locally compact Polish spaces $X$ and $Y$.

\begin{corollary}\label{col:twospaces}
Let $X,Y$ be locally compact Polish spaces. Let $\mathcal{K}$ be a convex set in $\mathcal{C}(X\cup Y)$ that is stable under maxima, contains constants and for any constant $c$ there is $\mathcal{K}+c\subset\mathcal{K}$. Let $f\in\mathcal{C}(X\cup Y)$. Suppose that $\mu\in\mathcal{P}(X)$ and $\nu\in\mathcal{P}(Y)$ are two Borel probability measures such that
\begin{equation}\label{eqn:majtwo}
\int_X g \,d\mu- \int_Yg \,d\nu\leq \int_X f\, d\mu-\int_Y f \,d\nu
\end{equation}
for all $g\in\mathcal{K}$.
Then there exists a Markov kernel $P$ from $X$ to $Y$ such that $\nu=P\mu$ and such that for $\mu$-almost every $x\in X$ and any $g\in\mathcal{K}$
\begin{equation*}
 \int_Y g \,d(\delta_x-P(\cdot,x))\leq \int_Y f \,d(\delta_x-P(\cdot,x)).
\end{equation*}
Moreover, the set of extreme points of the set of pairs of Borel probability measures that satisfy (\ref{eqn:majtwo})
is contained in the set of pairs of the form $(\delta_x,\eta)$ for some $x\in X$ and some Borel probability measure $\eta$ on $Y$.
\end{corollary}
\begin{proof}
Let $\Omega$ be the disjoint union of $X$ and $Y$. Let $\tilde{\mu},\tilde{\nu}$ be the probability Borel measures in $\mathcal{M}(\Omega)$ that are extensions of $\mu$ and $\nu$ respectively. Then (\ref{eqn:majtwo}) is equivalent to condition that for all $g\in\mathcal{K}$ there is 
\begin{equation*}
\int_{\Omega} g \,d(\tilde{\mu}- \tilde{\nu})\leq \int_{\Omega} f \,d(\tilde{\mu}- \tilde{\nu}).
\end{equation*}
Whence, by Theorem \ref{thm:extr} and Proposition \ref{pro:extr}, there exists a Markov kernel $\tilde{P}$ from $\Omega$ to $\Omega$ such that
\begin{equation*}
\tilde{\nu}=\tilde{P}\tilde{\mu}
\end{equation*}
and such that for every $\omega\in \Omega$ and for all $g\in\mathcal{K}$ there is
\begin{equation}\label{eqn:omega}
 \int_{\Omega}g \,d(\delta_{\omega}-\tilde{P}(\cdot,\omega))\leq \int_{\Omega} f \,d(\delta_{\omega}-\tilde{P}(\cdot,\omega)).
\end{equation}
Let $P(A,x)=\tilde{P}(A,x)$ for $x\in X$ and for any Borel set $A\subset Y$. Then $P$ is a Markov kernel from $X$ to $Y$. For this, observe that 
\begin{equation*}
1=\nu(Y)=\int_X \tilde{P}(Y,x)\,d\mu(x).
\end{equation*}
Hence, for $\mu$-almost every $x\in X$, $P(\cdot,x)$ is a Borel probability measure on $Y$. For $x\in X$ that belong to the complement of this measurable set, we put $P(\cdot,x)=\delta_{y_0}$. It follows that for every $x\in X$ there is $P(X,x)=0$. 
Moreover $\nu=P\mu$. By (\ref{eqn:omega}) it follows that for $\mu$-almost every $x\in X$ and all $g\in\mathcal{K}$ there is
\begin{equation*}
\int_Y g \,d(\delta_x- P(\cdot,x))\leq \int_Y f \,d(\delta_x- P(\cdot,x)).
\end{equation*}
The claim on the extreme points follows readily.
\end{proof}

\section{Martingale optimal transport}\label{s:martingale}

We shall characterise the set of extreme points of two Borel probability measures $\mu,\nu$ in convex order. Recall that two Borel probability measures $\mu,\nu$ on $\mathbb{R}^n$ with finite first moments are said to be in convex order provided that
\begin{equation*}
\int_{\mathbb{R}^n}g\,d\mu\leq\int_{\mathbb{R}^n}g\,d\nu
\end{equation*}
for all convex functions $g\colon\mathbb{R}^n\to\mathbb{R}$. 
Recall also that a set $K\subset\mathbb{R}^n$ is called a convex body provided that it is convex, compact and has non-empty interior.
We say that points $x_1,\dotsc,x_d\in\mathbb{R}^n$ are affinely independent provided that none of these points lies in the affine hull of the others.

\begin{theorem}\label{thm:ex}
Let $K\subset\mathbb{R}^n$ be a convex body. Let $\mathcal{F}$ denote the set of continuous convex functions on $K$. Let $\mathcal{P}$ denote the set of pairs $(\mu,\nu)$ of Borel probability measures on $K$ that are in convex order, that is 
\begin{equation*}
\int_K g\,d(\mu-\nu)\leq 0
\end{equation*}
for all $g\in\mathcal{F}$.
Then the set of extreme points of $\mathcal{P}$ is equal to the set of pairs of the form
\begin{equation}\label{eqn:form}
\Big(\delta_x,\sum_{i=1}^{d+1}\lambda_i\delta_{x_i}\Big)
\end{equation}
where $x=\sum_{i=1}^{d+1}\lambda_ix_i$, $\lambda_i>0$ for $i=1,\dotsc,d+1$ and $\sum_{i=1}^{d+1}\lambda_i=1$, $d\leq n$ and moreover $x_1,\dotsc,x_{d+1}\in K$ are affinely independent.
\end{theorem}
\begin{proof}
By Theorem \ref{thm:extr}, any extreme point of $\mathcal{P}$ is of the form $(\delta_x,\eta)$ for some $x\in K$ and some Borel probability measure $\eta$ on $K$. Moreover, as any affine function belongs to $\mathcal{F}$, we see that 
\begin{equation*}
x=\int_K y\,d\eta(y).
\end{equation*}
Let us fix $x\in K$.
Consider the set $\mathcal{A}$ of all Borel probability measures that have $x$ as their barycentre. To prove the assertion we ought to show that the extreme points of $\mathcal{A}$ are of the form 
\begin{equation*}
\sum_{i=1}^{d+1}\lambda_i\delta_{x_i}
\end{equation*}
for some positive $\lambda_1,\dotsc,\lambda_{d+1}$ that sum up to one, $d\leq n$ and $x_1,\dotsc,x_{d+1}$ affinely independent such that
\begin{equation*}
x=\sum_{i=1}^{d+1}\lambda_ix_i.
\end{equation*}
Let us first show that any extreme point $\gamma\in\mathcal{A}$ is supported on at most $n+1$ points. Suppose conversely, that there exist pairwise disjoint non-empty Borel sets $A_1,\dotsc,A_{n+2}\subset K$ such that 
\begin{equation*}
K=\bigcup_{i=1}^{n+2}A_i\text{ and }\gamma(A_i)>0\text{ for }i=1,\dotsc,n+2.
\end{equation*}
Then there exist real numbers $t_1,\dotsc,t_{n+2}$, not all of them equal, such that 
\begin{equation*}
0=\sum_{i=1}^{n+2}t_i \int_{A_i} (y-x)d\gamma(y).
\end{equation*}
We may assume that these numbers have absolute values all less than one and are such that
\begin{equation*}
-\frac12\leq \sum_{i=1}^{n+2}t_i\gamma(A_i)\leq \frac12.
\end{equation*}
Set 
\begin{equation*}
\gamma_1=\frac{\sum_{i=1}^{n+2}(1-t_i)\gamma|_{A_i}}{1-\sum_{i=1}^{n+2}t_i\gamma(A_i)}\text{ and }\gamma_2=\frac{\sum_{i=1}^{n+2}(1+t_i)\gamma|_{A_i}}{1+\sum_{i=1}^{n+2}t_i\gamma(A_i)}.
\end{equation*}
Then $\gamma_1,\gamma_2$ belong to $\mathcal{A}$. Moreover
\begin{equation*}
\gamma=\frac12\Big(1-\sum_{i=1}^{n+2}t_i\gamma(A_i)\Big)\gamma_1+\frac12\Big(1+\sum_{i=1}^{n+2}t_i\gamma(A_i)\Big)\gamma_2.
\end{equation*}
Thus $(\delta_x,\gamma)$ is not an extreme point of $\mathcal{A}$.
The contradiction yields that $\gamma$ is supported on at most $n+1$ points. 

Let $d+1\leq n+1$ be the number of points in the support. Let us show that we must necessarily have 
\begin{equation*}
\gamma=\sum_{i=1}^{d+1}\lambda_i\delta_{x_i}
\end{equation*}
for some positive numbers $\lambda_1,\dotsc,\lambda_{d+1}$ that sum up to one and $x_1,\dotsc,x_{d+1}$ affinely independent. Suppose that this is not the case. Then there exist
non-negative $\alpha_1,\dotsc,\alpha_{d+1}$, not all of them equal to $\lambda_1,\dotsc,\lambda_{d+1}$, such that
\begin{equation*}
x=\sum_{i=1}^{d+1}\alpha_ix_i\text{ and }\sum_{i=1}^{d+1}\alpha_i=1.
\end{equation*}
Set $\chi=\sum_{i=1}^{d+1}\alpha_i\delta_{x_i}$. Then $\chi\in\mathcal{A}$. Moreover, if $\epsilon\in\Big(0,\min\Big\{\frac{\lambda_1}{\alpha_1},\dotsc,\frac{\lambda_{d+1}}{\alpha_{d+1}},1\Big\}\Big)$, then we may write
\begin{equation*}
\gamma= \frac12(1-\epsilon)\frac{\gamma-\epsilon \chi}{1-\epsilon}+\frac12(1+\epsilon)\frac{\gamma+\epsilon\chi}{1+\epsilon},
\end{equation*}
as a convex combination of two distinct measures in $\mathcal{A}$.
This concludes the proof of the fact that any extreme point of $\mathcal{P}$ is of the form (\ref{eqn:form}).

Let us now show that any pair $\mu,\nu$ of that form is indeed an extreme point of $\mathcal{P}$. Observe that by Jensen's inequality any such pair belongs to $\mathcal{P}$. If we had
\begin{equation*}
(\mu,\nu)=\lambda(\theta_1,\rho_1)+(1-\lambda)(\theta_2,\rho_2)
\end{equation*}
for some $(\theta_1,\rho_1),(\theta_2,\rho_2)\in\mathcal{P}$ and some $\lambda\in (0,1)$,
then necessarily $\theta_1=\theta_2=\mu$, as $\mu$ is supported on a single point $x\in K$, and $\rho_1,\rho_2$ are supported on the support of $\nu$. As the points in the support of $\nu$ are affinely independent and 
\begin{equation*}
x=\int_{\Omega} y\,d \rho_1(y)=\int_{\Omega} y\, d\rho_2(y), 
\end{equation*}
we see that $\rho_1=\rho_2=\nu$.
\end{proof}

In the proof above we could have used a result of Winkler \cite{Winkler}. Instead we follow a direct approach for the sake of completeness and clarity.

The general method presented in this article may be applied to martingale optimal transport. In there one is given two Borel probability measures $\mu,\nu$ on a convex body $K\subset\mathbb{R}^n$ which are in convex order. The task is to find 
a coupling $\pi$ of $\mu$ and $\nu$ such that it is a distribution of a one-step martingale and that minimises the integral
\begin{equation*}
\int_{K\times K}c\,d\pi
\end{equation*}
among all such couplings. Here $c\colon K\times K\to\mathbb{R}$ is a given Borel measurable function, called a cost function. 

In the theorem below we shall employ the above characterisation of extreme points to prove a duality result for the multi-dimensional martingale optimal transport problem, provided that the value of the dual problem is attained. We also provide a novel formulation of both the primal and the dual problem.

For other results related to duality in the martingale optimal transport problem see results of Beiglb\"ock, Cox \cite{Beiglbock} and of Beiglb\"ock, Tim, Ob\l\'oj \cite{Beiglbock2}. 

Below we shall consider continuous functions $g\in\mathcal{C}(K\cup K)$ on the disjoint union of two copies of $K$. For such a function we shall denote by $g_1$ and $g_2$ the restrictions of $g$ to the first and to the second copy of $K$ respectively.

\begin{theorem}\label{thm:dualmart}
Let $K$ be a convex body in $\mathbb{R}^n$ and let $c\colon K\times K\to\mathbb{R}$ be a Lipschitz function. Let $\mu,\nu$ be two Borel probability measures on $K$ in convex order. Let $\mathcal{K}$ denote the set of continuous functions $g$ on the disjoint union of two copies of $K$ such that for all non-negative $\lambda_1,\dotsc,\lambda_{n+1}$ that add up to one and all $x_1,\dotsc,x_{n+1}\in K$ there is
\begin{equation*}
g_1\Big(\sum_{i=1}^{n+1}\lambda_i x_i\Big)-\sum_{i=1}^{n+1}\lambda_ig_2(x_i)\leq \sum_{i=1}^{n+1}\lambda_i c\Big(\sum_{j=1}^{n+1}\lambda_jx_j,x_i\Big).
\end{equation*}
Let $\mathcal{P}$ denote the set of pairs of Borel probability measures on $K$ that are in convex order.
Suppose that the supremum of integrals
\begin{equation*}
\int_Kf_1 \,d\mu-\int_K f_2\,d\nu
\end{equation*}
taken over the set $\mathcal{K}$ is attained. Then it is equal to the infimum of integrals
\begin{equation}\label{eqn:optimum}
\int_{\mathcal{E}}\int_{K\times K} c\,d(\xi_1\otimes \xi_2)\,d\pi(\xi)
\end{equation}
over the set of all Borel probability measures $\pi$ on the set $\mathcal{E}$ of extreme points of $\mathcal{P}$ such that
\begin{equation*}
(\mu,\nu)=\int_{\mathcal{E}}\xi \,d\pi(\xi).
\end{equation*}
Moreover the infimum is attained. It is also equal to the infimum of integrals
\begin{equation}\label{eqn:optimum2}
\int_{K\times K}\ c\,d\pi
\end{equation}
over all $\pi\in \Theta(\mu,\nu)$. Here $\Theta(\mu,\nu)$ stands for the set of all Borel probability measures on $K\times K$ such that its marginals are $\mu,\nu$ and that are distributions of a one-step martingale.
\end{theorem}

\begin{lemma}\label{lem:extrememart}
Let $K$ be a convex body in $\mathbb{R}^n$ and let $c\colon K\times K\to\mathbb{R}$ be a Lipschitz function.
Let $f\in \mathcal{K}$. 
Then the set of extreme points of the set $\mathcal{R}$ of pairs of Borel probability measures $(\mu,\nu)\in\mathcal{P}(K)\times\mathcal{P}(K)$ that are in convex order and such that
\begin{equation*}
\int_K g_1\,d\mu-\int_Kg_2\,d\nu\leq\int_K f_1\,d\mu-\int_Kf_2\,d\nu
\end{equation*}
for all $g\in\mathcal{K}$ is equal to the set of pairs of the form $(\delta_x,\sum_{i=1}^{d+1}\lambda_i x_i)$ for some $d\leq n$, $\lambda_1,\dotsc,\lambda_{d+1}$ positive that sum up to one, $x_1,\dotsc,x_{d+1}\in K$ affinely independent, $x=\sum_{i=1}^{d+1}\lambda_ix_i$, such that
\begin{equation*}
f_1\Big(\sum_{i=1}^{d+1}\lambda_i x_i\Big)-\sum_{i=1}^{d+1}\lambda_if_2(x_i)= \sum_{i=1}^{d+1}\lambda_ic\Big(\sum_{j=1}^{d+1}\lambda_jx_j,x_i\Big).
\end{equation*}
\end{lemma}
\begin{proof}
The set $\mathcal{K}$ is convex, stable under maxima, contains constants and for any $t\in\mathbb{R}$ there is $\mathcal{K}+t\subset\mathcal{K}$.
Thus, by Theorem \ref{thm:extr} and Proposition {pro:extr}, any extreme point of $\mathcal{R}$ is of the form $(\delta_{x_0},\eta)$ for some $x_0\in K$ and a Borel probability measure $\eta$ on $K$. Let $(\delta_{x_0},\eta)\in\mathcal{R}$ be such an extreme point. 
Let $h_0\in\mathcal{C}(K)$ be a convex continuous function on $K$. Let $h$ be a function equal to $h_0$ on the first copy of $K$ and equal to the same function $h_0$ on the other copy of $K$. Then $f+h\in\mathcal{K}$. Thus $\eta$ majorises $\delta_{x_0}$ in the convex order. Then we know that for any $g\in \mathcal{K}$ we have
\begin{equation}\label{eqn:bm}
\int_K \big(g_1(x_0)-g_2(y)\big)\,d\eta(y)\leq\int_K\big( f_1(x_0)-f_2(y)\big))\,d\eta(y).
\end{equation}
As $f\in\mathcal{K}$, the right-hand side of the above inequality is bounded above by 
\begin{equation}\label{eqn:c}
\int_K c(x_0,y)\,d\eta(y).
\end{equation}
Indeed, as $\eta$ majorises $\delta_{x_0}$ in the convex order, by Theorem \ref{thm:ex}, there exists a Borel probability measure on the set of extreme points $\mathcal{E}$ of $\mathcal{P}$ such that
\begin{equation}\label{eqn:pip}
(\delta_{x,_0}\eta)=\int_{\mathcal{E}}\xi \,d\pi(\xi).
\end{equation}
The fact that $f\in\mathcal{K}$ may be rephrased by
\begin{equation*}
\int_K f_1d\xi_1-\int_K f_2\,d\xi_2 \leq \int_{K\times K} c\,d(\xi_1\otimes \xi_2)
\end{equation*}
for all $\xi\in\mathcal{E}$.
The fact that (\ref{eqn:bm}) is bounded by (\ref{eqn:c}) follows by the integration against $\pi$.

By the McShane extension formula (see \cite{McShane}), we may assume that $c$ is defined and Lipschitz on $\mathbb{R}^n\times\mathbb{R}^n$.
Define $g$ so that for $y\in K$ we have $g_2(y)=-c(x_0,y)$ and for $x\in K$ set
\begin{equation*}
g_1(x)=\inf\Big\{ \sum_{i=1}^{n+1}\lambda_i \big(c(x,y_i)-c(x_0,y_i)\big)\mid \Big(\delta_x,\sum_{i=1}^{d+1}\lambda_i\delta_{y_i}\Big)\in \mathcal{E}\Big\}.
\end{equation*}
Here the infimum is over all pairs of measures in $\mathcal{E}$.
Then 
\begin{equation*}
g_1(x)-\sum_{i=1}^{n+1}\lambda_ig_2(y_i)\leq \sum_{i=1}^{n+1}\lambda_ic(x,y_i)
\end{equation*}
for all $y_1,\dotsc,y_{n+1}\in K$, all non-negative $\lambda_1,\dotsc,\lambda_{n+1}\geq 0$ summing up to one, with $x=\sum_{i=1}^{n+1}\lambda_iy_i$.  
Moreover, $g_1(x_0)=0$. We claim that $g$ is Lipschitz.
Indeed, for any $x,y\in K$ and any $y_1,\dotsc,y_{n+1}\in\mathbb{R}^n$ and non-negative $\lambda_1,\dotsc,\lambda_{n+1}$ that sum up to one and such that $y=\sum_{i=1}^{n+1}\lambda_iy_i$ we have
\begin{align*}
 &\sum_{i=1}^{n+1}\lambda_i \big(c(x,y_i+x-y))-c(x_0,y_i+x-y)\big)\leq\\
 &\leq   \sum_{i=1}^{n+1}\lambda_i \big(c(y,y_i)-c(x_0,y_i)\big)+3L\norm{x-y},
\end{align*}
where $L$ is the Lipschitz constant of $c$. Thus 
\begin{equation*}
g_1(x)\leq g_1(y)+3L\norm{x-y}.
\end{equation*}
This shows that $g_1$ is Lipschitz, hence continuous. In consequence, $g\in\mathcal{K}$.
Observe that, for such $g$, 
\begin{equation*}
\int_K \big(g_1(x_0)-g_2(y)\big)\,d\eta(y)=\int_K c(x_0,y)\,d\eta(y).
\end{equation*}
Let $\pi$ be as in (\ref{eqn:pip}). It follows by (\ref{eqn:bm}) that for $\pi$-almost every $\xi$ we have
\begin{equation*}
\int_K (f_1(x_0)-f_2(y)\big)\,d\xi_2(y)=\int_Kc(x_0,y)\,d\xi_2(y),
\end{equation*}
where $\delta_{x_0}=\xi_1$. Hence $\pi$-almost every $\xi\in\mathcal{E}$ belongs to $\mathcal{R}$.
Therefore any extreme point of $\mathcal{R}$ is necessary an extreme point of $\mathcal{P}$. The assertion follows readily.
\end{proof}

\begin{proof}[of Theorem \ref{thm:dualmart}]
First part of the theorem follows directly from Lemma \ref{lem:extrememart} and Choquet's theorem. For a proof of the second part, take a Borel probability measure $\pi_0$ on $\mathcal{E}$ that attains the infimum (\ref{eqn:optimum}). Set
\begin{equation*}
\pi=\int_{\mathcal{E}}\xi_1\otimes \xi_2 \,d\pi_0(\xi).
\end{equation*}
Then $\pi\in \Theta(\mu,\nu)$ is optimal for (\ref{eqn:optimum2}).
\end{proof}

\section{Dual problem in martingale optimal transport}\label{s:duality}

Let $\mu,\nu$ be Borel probability measures on $\mathbb{R}^n$ with finite first moments that are in convex order.
Martingale optimal transport problem between $\mu$ and $\nu$  admits a dual problem, which is to find the supremum of integrals
\begin{equation*}
\int_{\mathbb{R}^n} f_1 \,d\mu-\int_{\mathbb{R}^n}f_2 \,d\nu
\end{equation*}
taken over the set of all continuous functions $f_1,f_2\in\mathcal{C}(\mathbb{R}^n)$ such that 
\begin{equation*}
f_1(x)-f_2(y)\leq c(x,y)+\langle \gamma(x),y-x\rangle\
\end{equation*}
for all $x,y\in \mathbb{R}^n$ and for some map $\gamma\colon \mathbb{R}^n\to\mathbb{R}^n$. In this section we investigate this class of functions. We prove that this class of functions is equal to the class considered in the previous section.

\begin{definition}
Let $K\subset\mathbb{R}^n$ be a convex set. Then $F\subset K$ is a \emph{face} of $K$ if for any $z\in F$ and any $t\in (0,1)$ such that $z=tx+(1-t)y$ for some $x,y\in K$ we have $x,y\in F$.
\end{definition}

Let us observe that for any $x\in K$, there exists the minimal face of $K$ containing $x$. This holds true by the fact that $K$ itself is a face and that the intersection of a family of faces is a face. Uniqueness of minimal faces follows again by the property that the intersection of two faces is a face. Note that by the Hahn--Banach theorem it follows that $x$ belongs to the relative interior of the minimal face that contains $x$.

\begin{lemma}\label{lem:on}
Let $K$ be a convex body in $\mathbb{R}^n$. Let $c\colon K\times K\to\mathbb{R}$ be a continuous function. Let $f\in\mathcal{C}(K\cup K)$ be a continuous function on a disjoint union of two copies of $K$. Let $x\in K$. The following conditions are equivalent:
\begin{enumerate}[i)]
\item\label{i:fcone} for all $x_1,\dotsc,x_{n+1}\in K$ and all non-negative $\lambda_1,\dotsc,\lambda_{n+1}$ that sum up to one and such that $x=\sum_{i=1}^{n+1}\lambda_ix_i$ there is
\begin{equation*}
 f_1(x)-\sum_{i=1}^{n+1}\lambda_i f_2(x_i)\leq\sum_{i=1}^{n+1}\lambda_ic\Big(\sum_{j=1}^{n+1}\lambda_jx_j,x_i\Big),
\end{equation*}
\item\label{i:gammacone} there exists $\gamma\in\mathbb{R}^n$ such that for all $y\in K$ in the minimal face of $K$ that contains $x$ we have
\begin{equation*}
f_1(x)-f_2(y)\leq c(x,y)+\langle \gamma(x), y-x\rangle.
\end{equation*}
\item\label{i:martcone} for any random variable $X$ with values in $K$ and barycentre $x$ there is
\begin{equation*}
\mathbb{E}\big( f_1(x)-f_2(X)\big)\leq \mathbb{E}c(x,X).
\end{equation*}
\end{enumerate} 
\end{lemma}
\begin{proof}
Take any random variable $X$ with values in $K$ and barycentre $x$. Then there exists $\pi_0$ -- a Borel probability measure on the set $\mathcal{E}$ of extreme points of the set of pairs of Borel probability measures that have $x$ as their barycentre such that
\begin{equation*}
\int_{\mathcal{E}}\xi \,d\pi_0(\xi)
\end{equation*}
is a distribution of $X$. 
Note that, by Theorem \ref{thm:ex}, \ref{i:fcone} tells us that for any $\xi\in\mathcal{E}$ with there is 
\begin{equation*}
f_1(x)-\int_K f_2\,d\xi\leq \int_K c(x,\cdot) \, d\xi.
\end{equation*}
Therefore
\begin{align*}
\mathbb{E}\big(f_1(x)-f_2(X)\big)&=\int_{\mathcal{E}}\Big( f_1(x)-\int_K f_2\,d\xi  \Big)\,d\pi_0(\xi)\leq\\
&\leq \int_{\mathcal{E}} \int_K c(x,\cdot)\,d \xi \,d\pi_0(\xi)=\mathbb{E} c(x,X).
\end{align*}
This is to say, \ref{i:fcone} implies \ref{i:martcone}. 

For the converse implication, take any $x_1,\dotsc,x_{n+1}\in K$ and any non-negative $\lambda_1,\dotsc,\lambda_{n+1}$ that sum up to one with $x=\sum_{i=1}^{n+1}\lambda_ix_i$. Let $X=x_i$ with probability $\lambda_i$ for $i=1,\dotsc,n+1$. Then \ref{i:fcone} follows from an application of \ref{i:martcone} for the random variable $X$.

For the proof of the other equivalences, without loss of generality, we may assume that $c$ is non-negative. Let us denote the set of continuous functions on $K\cup K$ that satisfy condition \ref{i:fcone}  by $\mathcal{K}_1$ and the set of continuous functions on $K\cup K$ that satisfy condition \ref{i:gammacone} by $\mathcal{K}_2$. Observe that trivially $\mathcal{K}_2\subset \mathcal{K}_1$ and that $\mathcal{K}_2$ is a closed, convex set in $\mathcal{C}(K\cup K)$.
%

Suppose that there exists $f\in \mathcal{K}_1\setminus \mathcal{K}_2$. Then by the Hahn--Banach theorem there exists a Borel measure $\eta\in\mathcal{M}(K\cup K)$ such that for all $g\in\mathcal{K}_2$
\begin{equation}\label{eqn:assuc}
\int_{K\cup K} f\,d\eta> \int_{K\cup K} g\,d\eta.
\end{equation}
Since constant functions belong to $\mathcal{K}_2$, measure $\eta$ may be written as a difference of two Borel measures of equal masses. Without loss of generality we may assume that these measures are probabilities. Since any continuous function that is non-positive at $\{x\}$ on the first copy of $K$ and non-negative on the second copy of $K$ belongs to $\mathcal{K}_2$, we see that $\eta=\delta_x-\nu$, for some non-negative $\nu$ supported on the second copy of $K$. Here $\delta_x$ is supported on the first copy of $K$.

Observe that if $h_0\in\mathcal{C}(K)$ is any convex function then $h\in\mathcal{C}(K\cup K)$ such that $h_1=h_2=h_0$ belongs to $\mathcal{K}_2$. Thus $\delta_x$ and $\nu$ are in convex order, i.e., $x$ is the barycentre of $\nu$.

Define $k\colon K\cup K\to\mathbb{R}$ by 
$k_2(y)=-c(x,y)$ for $y\in K$  and for $z\in K$ set
\begin{equation*}
k_1(z)=0.
\end{equation*} 
Then $k$ is continuous by continuity of $c$ and thus $k\in \mathcal{K}_2$, with $\gamma$ equal to zero.  It follows that 
\begin{equation}\label{eqn:bouc}
 \int_K c(x,y) \,d\nu(y)=\int_K \big(k_1(x)-k_2\big)\,d\nu< \int_K \big(f_1(x)-f_2\big)\,d\nu.
\end{equation}
Recall that $(\delta_x,\nu)$ has barycentre $x$. Thus, there exists a probability measure $\pi$ on the set $\mathcal{E}$ of extreme points of measures with barycentre $x$ such that
\begin{equation*}
(\delta_x,\nu)=\int_{\mathcal{E}}\xi \,d\pi(\xi).
\end{equation*}
It follows, by Theorem \ref{thm:ex}, and the definition of $\mathcal{K}_1$, that
\begin{equation*}
\int_K \big(f_1(x)-f_2\big)\,d\nu \leq \int_K  c(x,y)\,d\nu(y).
\end{equation*}
This stands in contradiction to (\ref{eqn:bouc}) and proves that $f\in\mathcal{K}_2$. This is to say $\mathcal{K}_1=\mathcal{K}_2$.
\end{proof}

The above lemma implies also the characterisation for the case of arbitrary convex set $K\subset\mathbb{R}^n$, not necessarily a compact one.

\begin{theorem}\label{thm:col}
Let $K$ be a convex set in $\mathbb{R}^n$. Suppose that $c\colon K\times K\to\mathbb{R}$ is a continuous function. Let $f$ be a continuous function on a disjoint union $K\cup K$. Let $x\in K$. The following conditions are equivalent:
\begin{enumerate}[i)]
\item\label{i:fcc} for all $x_1,\dotsc,x_{n+1}\in K$ and all non-negative $\lambda_1,\dotsc,\lambda_{n+1}$ that sum up to one and such that $x=\sum_{i=1}^{n+1}\lambda_ix_i$ there is
\begin{equation*}
 f_1(x)-\sum_{i=1}^{n+1}\lambda_i f_2(x_i)\leq\sum_{i=1}^{n+1}\lambda_ic\Big(\sum_{j=1}^{n+1}\lambda_jx_j,x_i\Big),
\end{equation*}
\item\label{i:gammacc} there exists $\gamma\in\mathbb{R}^n$ such that for all $y\in K$ in the minimal face of $K$ that contains $x$ we have
\begin{equation*}
f_1(x)-f_2(y)\leq c(x,y)+\langle \gamma(x), y-x\rangle.
\end{equation*}
\item\label{i:martcc} for any bounded random variable $X$ with values in $K$ and barycentre $x$ there is
\begin{equation*}
\mathbb{E}\big( f_1(x)-f_2(X)\big)\leq \mathbb{E}c(x,X).
\end{equation*}
\end{enumerate} 
\end{theorem}
\begin{proof}
Assume first that $x\in\mathrm{int}K$.
Choose a increasing sequence $(K_n)_{n=1}^{\infty}$ of compact convex subsets of $K$ such that its union is $\mathrm{int}K$. Suppose that $f$ is continuous and satisfies \ref{i:fcc}. Let $\epsilon>0$ be such that $B(x,\epsilon)\subset \mathrm{int}K$. Here $B(x,\epsilon)$ denotes the closed ball of radius $\epsilon$ centred at $x$. Then by Lemma \ref{lem:on} for any $n\in\mathbb{N}$ sufficiently large so that $x\in\mathrm{int}K_n$ there exists $\gamma_n$ such that for all $y\in K_n$ there is
\begin{equation}\label{eqn:costam}
f_1(x)-f_2(y)\leq c(x,y)+\langle\gamma_n, y-x\rangle.
\end{equation}
Let $n_0$ be such that $B(x,\epsilon)\subset K_{n_0}$. Take $n>n_0$. Suppose that $\gamma_n\neq 0$ and set $y_n=x-\epsilon\frac{\gamma_n}{\norm{\gamma_n}}$. Then $y_n\in K_{n_0}\subset K_n$ and therefore, by (\ref{eqn:costam}),
\begin{equation*}
\norm{\gamma_n}\leq\frac1{\epsilon}\big( c(x,y_n)-f_1(x)+f_2(y_n)\big).
\end{equation*}
As $c$ is bounded on $\{x\}\times K_{n_0}$ and $f$ is bounded on $K_{n_0}$, the right-hand side of the above inequality is bounded. Hence, so is the left-hand side. We may therefore pick $\gamma$ that is an accumulation point of the sequence $(\gamma_n)_{n=1}^{\infty}$. From (\ref{eqn:costam}) and from continuity of $f$ it follows now that for all $y\in K$
\begin{equation*}
f_1(x)-f_2(y)\leq c(x,y)+\langle\gamma, y-x\rangle.
\end{equation*}
This is to say, $f$ satisfies also \ref{i:gammacc} if $x\in\mathrm{int}K$. If $\gamma_n=0$ for infinitely many $n$, then the above inequality holds true with $\gamma=0$.

In general case, let $L$ denote the minimal face of $K$ that contains $x$. Then $x$ belongs to the relative interior of $L$. We repeat the above argument with $K$ replaced by $L$ considered as a convex subset of its affine hull.

That \ref{i:gammacc} implies \ref{i:fcc} is straightforward.

The equivalence of \ref{i:fcc} and \ref{i:martcc} follows readily from Lemma \ref{lem:on}.
\end{proof}

\begin{remark}
We see that the classes of functions considered in Section \ref{s:martingale} are characterised by the fact that corresponding conditions \ref{i:fcone} hold true for any $x\in K$. 
\end{remark}

\section{Martingale triangle inequality}\label{s:triangle}

In this section we introduce the notion of \emph{martingale triangle inequality} for cost functions $c\colon K\times K\to\mathbb{R}$, where $K\subset\mathbb{R}^n$ is a convex set. We shall show that if it is satisfied by a cost function $c$, which vanishes on the diagonal, then one may take $f_1=f_2$ in the dual problem to the martingale optimal transport. 

\begin{definition}\label{defin:martingale}
Let $K\subset\mathbb{R}^n$ be a convex set. Let $c\colon K\times K\to\mathbb{R}$. We say that $c$ satisfies \emph{martingale triangle inequality}  provided that for all $x,x_1,\dotsc,x_{n+1}\in K$ and all non-negative $\lambda_1,\dotsc,\lambda_{n+1}$ that sum up to one there is
\begin{equation}\label{eqn:con}
\sum_{i=1}^{n+1}\lambda_i c(x,x_i)-c\Big(x,\sum_{i=1}^{n+1}\lambda_ix_i\Big)\leq \sum_{i=1}^{n+1} \lambda_i c\Big( \sum_{j=1}^{n+1}\lambda_j x_j,x_i\Big).
\end{equation}
\end{definition}

In other words, for any $x\in K$ function $-c(x,\cdot)$ satisfies condition \ref{i:fccc} of Theorem \ref{thm:col}, with cost function $c$.

\begin{remark}
Condition (\ref{eqn:con}) is satisfied if $c$ is a metric on $K$ and also it is satisfied if $c$ is concave in the second variable and non-negative. Also, this condition defines a closed convex cone of functions.
Note also that for a function given by $c(x,y)=\norm{x-y}^2$ for $x,y\in K$, where $\norm{\cdot}$ denotes Euclidean norm on $K$, we have equality in (\ref{eqn:con}).
\end{remark}

The following theorem is a martingale optimal transport analogue of Kantorovich--Rubinstein duality for the classical optimal transport problem with a metric cost function, see Section \ref{s:Kant-Rub}.

\begin{theorem}\label{thm:clc}
Let $K$ be a convex body in $\mathbb{R}^n$. Let $c\colon K\times K\to\mathbb{R}$ be a continuous function satisfying martingale triangle inequality and vanishing on the diagonal.

Let $\mathcal{B}_1$ denote the set of functions $f\in\mathcal{C}(K)$ such that for all $x_1,\dotsc,x_{n+1}\in K$ and all non-negative $\lambda_1,\dotsc,\lambda_{n+1}$ that sum up to one there is
\begin{equation*}
f\Big(\sum_{i=1}^{n+1}\lambda_ix_i\Big)- \sum_{i=1}^{n+1}\lambda_i f(x_i)\leq\sum_{i=1}^{n+1}\lambda_ic\Big(\sum_{j=1}^{n+1}\lambda_jx_j,x_i\Big).
\end{equation*}

Let $\mathcal{B}_2$ denote the set of functions $g\in\mathcal{C}( K\cup K)$ on the disjoint union of two copies of $K$ such that for all $x_1,\dotsc,x_{n+1}\in K$ and all non-negative $\lambda_1,\dotsc,\lambda_{n+1}$ that sum up to one there is
\begin{equation*}
g_1\Big(\sum_{i=1}^{n+1}\lambda_ix_i\Big)- \sum_{i=1}^{n+1}\lambda_i g_2(x_i)\leq\sum_{i=1}^{n+1}\lambda_ic\Big(\sum_{j=1}^{n+1}\lambda_jx_j,x_i\Big).
\end{equation*}
Then, for any Borel probability measures $\mu,\nu$ on $K$ in convex order there is
\begin{equation}\label{eqn:supremumm}
\sup\Big\{\int_K f \,d(\mu-\nu)\mid f\in \mathcal{B}_1\Big\}=\sup\Big\{\int_K g_1 \,d\mu-\int_K g_2 \,d\nu\mid g\in \mathcal{B}_2\Big\}.
\end{equation}
\end{theorem}
\begin{proof}
Clearly, the supremum on the right-hand side of (\ref{eqn:supremumm}) is at least the supremum on the left-hand side of (\ref{eqn:supremumm}), as if $f\in\mathcal{B}_1$, then $g\colon K\cup K\to\mathbb{R}$ defined by $g_1=f$ and $g_2=f$ belongs to $\mathcal{B}_2$.

Suppose that there exists $\epsilon>0$ such that for any function $f\in\mathcal{B}_1$
\begin{equation*}
\int_K f \,d(\mu-\nu)+2\epsilon\leq \sup\Big\{\int_K g_1\, d\mu-\int_K g_2 \,d\nu\mid g\in \mathcal{B}_2\Big\}.
\end{equation*}
It follows that there exists $g\in\mathcal{B}_2$ such that for all $f\in\mathcal{B}_1$ we have 
\begin{equation}\label{eqn:eps}
\int_K f \,d(\mu-\nu)+\epsilon\leq \int_K g_1 \,d\mu-\int_K g_2 \,d\nu.
\end{equation}
By Corollary \ref{col:twospaces}, any extreme point of pairs of probability measures $\mathcal{P}$ that satisfy 
\begin{equation*}
\int_K f \,d(\mu-\nu)\leq \int_K g_1\, d\mu-\int_K g_2 \,d\nu.
\end{equation*}
for all $f\in\mathcal{B}_1$ is of the form $(\delta_{x_0},\eta)$ for some $x_0\in K$ and a Borel probability measure $\eta$. Hence, by (\ref{eqn:eps}), for some $x_0\in K$ and some $\eta\in\mathcal{P}(K)$ and all $f\in\mathcal{B}_1$
\begin{equation}\label{eqn:epsprim}
\int_K(f(x_0)- f)\,d\eta+\epsilon\leq \int_K \big( g_1(x_0)-g_2\big)\,d\eta.
\end{equation}
Note that any such pair $(\delta_{x_0},\eta)$ is in convex order. Thus there exists a Borel probability measure $\pi_0$ on the set $\mathcal{E}$ of extreme points of pairs of measures in convex order such that 
\begin{equation}\label{eqn:repres}
(\delta_{x_0},\eta)=\int_{\mathcal{E}}\xi \,d\pi_0(\xi).
\end{equation}
But, as $g\in\mathcal{B}_2$, for any $\xi\in\mathcal{E}$ we have 
\begin{equation*}
\int_{K\times K} (g_1(x)-g_2(y))\,d(\xi_1\otimes\xi_2)(x,y)\leq \int_{K\times K} c\, d(\xi_1\otimes \xi_2).
\end{equation*}
Take $f=-c(x_0,\cdot)$. Then, by the martingale triangle inequality, $f\in\mathcal{B}_1$. This, together wih (\ref{eqn:repres}) and (\ref{eqn:epsprim}), yields a contradiction. 
\end{proof}

\section{Continuity properties}\label{s:continuity}

We shall now investigate continuity properties of the class of functions considered in Theorem \ref{thm:clc}. We adapt the argument of \cite{Wayne}. We shall need the following lemma.

\begin{lemma}\label{lem:onedim}
Suppose that $f\colon [a,b]\to\mathbb{R}$ is such that for all $\lambda\in [0,1]$ and all $x,y\in [a,b]$ there is
\begin{equation*}
\lambda f(x)+(1-\lambda)f(y)-f(\lambda x+(1-\lambda)y)\leq \lambda c(\lambda x+(1-\lambda)y, x)+(1-\lambda)c(\lambda x+(1-\lambda)y,y).
\end{equation*}
Then for all $a\leq x_1<x_2<x_3\leq b$ the quotient $\frac{f(x_3)-f(x_1)}{x_3-x_1}$ is bounded below by 
\begin{equation*}
\frac{f(x_3)-f(x_2)}{x_3-x_2}+\frac{c(x_2,x_3)-c(x_2,x_1)}{x_3-x_1}-\frac{c(x_2,x_3)}{x_3-x_2}
\end{equation*}
and above by
\begin{equation*}
\frac{f(x_2)-f(x_1)}{x_2-x_1}+\frac{c(x_2,x_3)-c(x_2,x_1)}{x_3-x_1}+\frac{c(x_2,x_1)}{x_2-x_1}.
\end{equation*}
\end{lemma}
\begin{proof}
Let $\lambda\in (0,1)$ be such that $x_2=\lambda x_1+(1-\lambda)x_3$, that is 
\begin{equation*}
\lambda=\frac{x_3-x_2}{x_3-x_1}.
\end{equation*}
Then we know that
\begin{equation*}
\lambda f(x_1)+(1-\lambda)f(x_3)-f(x_2)\leq \lambda c(x_2,x_1)+(1-\lambda)c(x_2,x_3).
\end{equation*}
Hence putting formula for $\lambda$ we obtain that
\begin{equation*}
\frac{f(x_3)-f(x_2)}{x_3-x_2}+\frac{c(x_2,x_3)-c(x_2,x_1)}{x_3-x_1}-\frac{c(x_2,x_3)}{x_3-x_2}\leq \frac{f(x_3)-f(x_1)}{x_3-x_1}
\end{equation*}
and
\begin{equation*}
\frac{f(x_3)-f(x_1)}{x_3-x_1}\leq\frac{f(x_2)-f(x_1)}{x_2-x_1}+\frac{c(x_2,x_3)-c(x_2,x_1)}{x_3-x_1}+\frac{c(x_2,x_1)}{x_2-x_1}.
\end{equation*}
\end{proof}

\begin{theorem}\label{thm:lipm}
Let $K$ be a convex, open set in $\mathbb{R}^n$. Suppose that $f\colon K\to\mathbb{R}$ is such that for all $x,y\in K$ and all $\lambda\in [0,1]$ there is
\begin{equation*}
\lambda f(x)+(1-\lambda)f(y)-f(\lambda x+(1-\lambda)y)\leq  \lambda c(\lambda x+(1-\lambda)y, x)+(1-\lambda)c(\lambda x+(1-\lambda)y,y).
\end{equation*}
Suppose that $c$ is $L$-Lipschitz in the second variable and is such that for all $x,y\in K$ there is $\abs{c(x,y)}\leq \Lambda\norm{x-y}$ for some constant $\Lambda$. Then $f$ is locally Lipschitz in $K$.
\end{theorem}
\begin{proof}
Suppose that $n=1$. Then, without loss of generality, $K=[a,d]$ for some $a<d$. Choose numbers $b,c$ so that $a<b<c<d$. Then applying Lemma \ref{lem:onedim} four times yields that for any $x,y$ such that $b<x<y<c$ we have
\begin{equation*}
\frac{f(y)-f(x)}{y-x}\leq \frac{f(b)-f(a)}{b-a}+\frac{c(x,y)}{y-x}+\frac{c(b,a)}{b-a}+\frac{c(b,y)-c(b,a)}{y-a}-\frac{c(x,y)-c(x,a)}{y-a}
\end{equation*}
and
\begin{equation*}
\frac{f(d)-f(c)}{d-c}-\frac{c(c,x)}{c-x}+\frac{c(c,d)-c(c,x)}{d-x}-\frac{c(y,x)}{y-x}-\frac{c(y,d)-c(y,x)}{d-x}\leq\frac{f(y)-f(x)}{y-x}.
\end{equation*}
In particular on $[b,c]$ function $f$ has Lipschitz constant at most
\begin{equation*}
\max\Big\{\Big|\frac{f(b)-f(a)}{b-a}+2L+2\Lambda\Big|,\Big|\frac{f(d)-f(c)}{d-c}-2L-2\Lambda\Big|\Big\}.
\end{equation*}
Suppose now that $n>1$ and that, by induction, the lemma holds true for all dimensions at most $n-1$. Choose any simplices $X,Y$ and any ball $B$ in $K$ and such that $B\subset X\subset Y\subset K$ and such that $B$ and the boundaries of $X$ and $Y$ are pairwise disjoint. Then, by the inductive assumption, $f$ is continuous on the boundaries of $X$ and $Y$, and therefore the function
\begin{equation*}
\partial X\times \partial Y\ni (x,y)\mapsto\frac{\abs{f(x)-f(y)}}{\norm{x-y}}\in\mathbb{R}
\end{equation*}
is bounded by a constant $M$. Choose any points $x,y\in B$. Choose a unique line passing through $x$ and $y$. Then there exist unique points $x_1,x_2\in X$ and $y_1,y_2\in Y$ such that the line intersects $\partial X$ in $x_1,x_2$ and $\partial Y$ in $y_1,y_2$ where, without loss of generality, 
\begin{equation*}
y_1<x_1<x<y<x_2<y_2
\end{equation*}
on the line. By Lemma \ref{lem:onedim} we see that 
\begin{equation*}
\frac{\abs{f(y)-f(x)}}{\norm{x-y}}\leq\max\Big\{\Big|\frac{f(y_2)-f(x_2)}{\norm{y_2-x_2}}-2L-2\Lambda\Big|,\Big|\frac{f(x_1)-f(y_1)}{\norm{y_2-x_2}}+2L+2\Lambda\Big|\Big\}\
\end{equation*}
Therefore $f$ has Lipschitz constant at most $M+2L+2\Lambda$ on $B$.
\end{proof}

\begin{theorem}\label{thm:onefunction}
Let $K$ be a convex set in $\mathbb{R}^n$. Suppose that $c\colon K\times K\to\mathbb{R}$ is a continuous function. Let $x\in K$ and let $f$ be a continuous function on $K$. The following conditions are equivalent:
\begin{enumerate}[i)]
\item\label{i:fccc} for all $x_1,\dotsc,x_{n+1}\in K$ and all non-negative $\lambda_1,\dotsc,\lambda_{n+1}$ that sum up to one and such that $x=\sum_{i=1}^{n+1}\lambda_ix_i$ there is
\begin{equation*}
f(x)- \sum_{i=1}^{n+1}\lambda_i f(x_i)\leq\sum_{i=1}^{n+1}\lambda_ic\Big(\sum_{j=1}^{n+1}\lambda_jx_j,x_i\Big),
\end{equation*}
\item\label{i:gammaccc} there exists $\gamma\in\mathbb{R}^n$ such that for all $y\in K$ in the minimal face of $K$ that contains $x$ we have
\begin{equation*}
f(x)-f(y)\leq c(x,y)+\langle \gamma(x), y-x\rangle.
\end{equation*}
\item\label{i:martccc} for any bounded random variable $X$ with values in $K$ and barycentre $x$ there is
\begin{equation*}
\mathbb{E}\big( f(x)-f(X)\big)\leq \mathbb{E}c(x,X).
\end{equation*}
\end{enumerate} 
Moreover, suppose that $K$ is open and additionally for any $x_0\in K$ there exist open, convex set $K'\subset K$ such that $x_0\in K'$ and $\abs{c(x,y)}\leq\Lambda \norm{x-y}$ for all $x,y\in K'$ and some constant $\Lambda$ and that $c$ is locally Lipschitz in the second variable. Suppose that $f\colon K\to\mathbb{R}$ satisfies one of the above conditions for each $x\in K$. Then it satisfies all the other conditions for each $x\in K$.
 In such case, any function that satisfies the above conditions is locally Lipschitz.
\end{theorem}
\begin{proof}
The equivalence of the conditions \ref{i:fccc} \ref{i:gammaccc} and \ref{i:martccc} follows from Theorem \ref{thm:col}. The second part of the corollary follows from Theorem \ref{thm:lipm}, as, in such case, any function $f$ that satisfies \ref{i:fccc}, \ref{i:gammaccc} or \ref{i:martccc} is continuous in $K$.
\end{proof}

\section{Uniform convexity and uniform smoothness}\label{s:uniform}

In this section we employ the results of Section \ref{s:duality} and Section \ref{s:continuity} to provide a characterisation of uniformly smooth and uniformly convex functions on $\mathbb{R}^n$, or, more generally, on an open, convex set $K\subset\mathbb{R}^n$. We refer the reader to \cite{Aze} and to \cite{Zalinescu} and references therein for previous studies of the topic. Let us recall the definitions. 

\begin{definition}
Let $\sigma\colon\mathbb{R}\to\mathbb{R}$. A function $f\colon K\to\mathbb{R}$ is called $\sigma$-convex provided that 
\begin{equation*}
f(\lambda x+(1-\lambda)y)+\lambda(1-\lambda)\sigma(\norm{x-y})\leq \lambda f(x)+(1-\lambda)f(y)
\end{equation*}
for all $\lambda\in [0,1]$ and all $x,y\in K$. A function $g\colon K\to\mathbb{R}$ is called $\sigma$-smooth provided that 
\begin{equation*}
g(\lambda x+(1-\lambda)y)+\lambda(1-\lambda)\sigma(\norm{x-y})\geq \lambda g(x)+(1-\lambda)g(y)
\end{equation*}
for all $\lambda\in [0,1]$ and all $x,y\in K$. 
\end{definition}

Another notion of convexity and smoothness is as follows, see \cite{Aze}. 

\begin{definition}
Let $\gamma\in\mathbb{R}^n$ and let $x\in K$. We say that $f\colon K\to\mathbb{R}$ is $\sigma$-uniformly convex at $x$ with respect to $\gamma$ if for all $y\in\ K$ there is
\begin{equation*}
f(x)+\sigma(\norm{y-x})+\langle \gamma,y-x\rangle\leq f(y).
\end{equation*}
Likewise, $g\colon K\to\mathbb{R}$ is called $\sigma$-uniformly smooth at $x$ with respect to $\gamma$ if for all $y\in K$ there is
\begin{equation*}
g(x)+\sigma(\norm{y-x})+\langle \gamma,y-x\rangle\geq g(y).
\end{equation*}
\end{definition}

Note that the condition that $f\colon K\to\mathbb{R}$ is $\sigma$-uniformly convex at $x\in K$ is  equivalent to condition \ref{i:gammaccc} of Theorem \ref{thm:onefunction} for the function 
\begin{equation*}
c(x,y)=-\sigma(\norm{y-x})\text{, }x,y\in K. 
\end{equation*}
Similarly, $\sigma$-uniform smoothness at a point $x\in K$ of a function $g\colon K\to\mathbb{R}$ is equivalent to condition \ref{i:gammaccc} of Theorem \ref{thm:onefunction} for $-g$ and the function 
\begin{equation*}
c(x,y)=\sigma(\norm{y-x})\text{, } x,y\in K.
\end{equation*}

Now, Theorem \ref{thm:onefunction} implies the following theorem, which complements the results of \cite{Aze}.

\begin{theorem}\label{thm:uni}
Let $K\subset\mathbb{R}^n$ be an open, convex set. Suppose that $\sigma\colon\mathbb{R}\to\mathbb{R}$ is locally Lipschitz function such that $\sigma(0)=0$.
Let $f\colon K\to\mathbb{R}$. The following conditions are equivalent:
\begin{enumerate}[i)]
\item there exists $\gamma\colon K\to\mathbb{R}^n$ such that for any $x\in K$ the function $f$ is $\sigma$-uniformly convex at $x$ with respect to $\gamma(x)\in\mathbb{R}^n$,
\item for any $x_1,\dotsc,x_{n+1}\in K$ and any non-negative $\lambda_1,\dotsc,\lambda_{n+1}$ that sum up to one there is 
\begin{equation*}
f\Big(\sum_{i=1}^{n+1}\lambda_ix_i\Big)-\sum_{i=1}^{n+1}\lambda_i f(x_i)\leq -\sum_{i=1}^{n+1}\lambda_i\sigma\Big(\Big\lVert \sum_{j=1}^{n+1}\lambda_jx_j-x_i \Big\rVert\Big),
\end{equation*}
\item for any bounded one-step martingale $(X_0,X_1)$ with values in $K$ there is
\begin{equation*}
\mathbb{E}\big(f(X_0)-f(X_1)\big)\leq -\mathbb{E}\sigma(\norm{X_0-X_1}).
\end{equation*}
\end{enumerate}
Also,  the following conditions are equivalent:
\begin{enumerate}[i)]
\item there exists $\gamma\colon K\to\mathbb{R}^n$ such that for any $x\in K$ the function $f$ is $\sigma$-uniformly smooth at $x$ with respect to $\gamma(x)\in\mathbb{R}^n$,
\item for any $x_1,\dotsc,x_{n+1}\in K$ and any non-negative $\lambda_1,\dotsc,\lambda_{n+1}$ that sum up to one there is 
\begin{equation*}
f\Big(\sum_{i=1}^{n+1}\lambda_ix_i\Big)-\sum_{i=1}^{n+1}\lambda_i f(x_i)\geq -\sum_{i=1}^{n+1}\lambda_i\sigma\Big(\Big\lVert \sum_{j=1}^{n+1}\lambda_jx_j-x_i \Big\rVert\Big),
\end{equation*}
\item for any bounded one-step martingale $(X_0,X_1)$ with values in $K$ there is
\begin{equation*}
\mathbb{E}\big(f(X_0)-f(X_1)\big)\geq -\mathbb{E}\sigma(\norm{X_0-X_1}).
\end{equation*}
\end{enumerate}
Moreover, any function $f$ that satisfies one of the above conditions is locally Lipschitz in $K$.

If we assume that $\sigma$ and $f$ are continuous, then the respective conditions at single $x\in K$ are equivalent.
\end{theorem}
\begin{proof}
The assumptions on $\sigma$ imply that the functions
\begin{equation*}
(x,y)\mapsto-\sigma(\norm{y-x})\text{ and }(x,y)\mapsto \sigma(\norm{y-x})
\end{equation*}
are locally Lipschitz in $K\times K$ and moreover for any $x_0\in K$ there exists an open, convex set $K'\subset K$ such that $x_0\in K'$ and for all $x,y\in K'$ there is 
\begin{equation*}
\abs{\sigma(\norm{y-x})}\leq\Lambda\norm{y-x}, 
\end{equation*}
where $\Lambda$ depends on $\sigma$. Therefore, from Theorem \ref{thm:lipm}, we infer that any function that satisfies one of the above conditions is continuous. Hence, the assumptions of Theorem \ref{thm:onefunction} are satisfied. 
Thus, the assertion of the theorem follows from the conclusion of Theorem \ref{thm:onefunction}.

If we do not put any assumptions on regularity of $\sigma$, but its continuity, then the claim follows again by Theorem \ref{thm:onefunction}.
\end{proof}

\section{Optimal transport}\label{s:transport}

In this section, Corollary \ref{col:twospaces} is employed to prove Kantorovich duality. The theorem below also provides a reinterpretation of the Kantorovich problem as minimisation of a linear functional over all Choquet's representation of a pair of probability measures.

\begin{theorem}\label{thm:opti}
Let $X,Y$ be two locally compact Polish spaces and let $c\colon X\times Y\to\mathbb{R}$ be a bounded Lipschitz function. Let $\mu$ and $\nu$ be Borel probability measures on $X$ and $Y$ respectively.
Then the supremum of integrals
\begin{equation*}
\int_X \phi \,d\mu-\int_Y \psi \,d\nu
\end{equation*}
taken over the set of continuous, bounded functions $\phi\in\mathcal{C}(X),\psi\in\mathcal{C}(Y)$ such that
\begin{equation*}
 \phi(x)-\psi(y)\leq c(x,y)\text{ for all }x\in X,y\in Y
\end{equation*}
is equal to the infimum of integrals
\begin{equation*}
\int_{X\times Y} c\,d\pi 
\end{equation*}
over all $\pi\in \Gamma (\mu,\nu)$. Here $\Gamma(\mu,\nu)$ stands for the set of all Borel probability measures on $X\times Y$ such that its marginals are $\mu$ and $\nu$ respectively. Moreover, both supremum and infimum are attained.
\end{theorem}

Before we come to the proof of the above theorem, let us first recall that the supremum in the statement is attained. We refer to  \cite[Proof of Theorem 3.1]{Villani2} for a detailed proof.

\begin{lemma}\label{lem:attains}
There exist $\phi_0\in\mathcal{C}(X)$ and $\phi_0\in\mathcal{C}(Y)$ such that for all $x\in X$ and $y\in Y$ there is $\phi_0(x)-\psi_0(y)\leq c(x,y)$ and for all $\phi\in\mathcal{C}(X)$ and all $\psi\in\mathcal{C}(Y)$ that satisfy $\phi(x)-\psi(y)\leq c(x,y)$ for all $x\in X$ and $y\in Y$ there is
\begin{equation*}
\int_X \phi \,d\mu-\int_Y \psi \,d\nu\leq \int_X \psi_0 \,d\mu-\int_Y \phi_0 \,d\nu.
\end{equation*}
\end{lemma}

\begin{proof}[of Theorem \ref{thm:opti}]
Without loss of generality we may assume that $c$ is non-negative. Pick $\phi_0\in\mathcal{C}(X)$ and $\psi_0\in\mathcal{C}(Y)$ from the lemma above.
Define $\rho_0\in\mathcal{C}(X\cup Y)$ so that $\rho_0(x)=\phi_0(x)$ for $x\in X$ and $\rho_0(y)=\psi_0(y)$ for $y\in Y$.
Let $\mathcal{K}$ denote the set of all bounded continuous functions $\rho$ on $X\cup Y$ such that 
for $x\in X$ and $y\in Y$ there is
\begin{equation*}
\rho(x)-\rho(y)\leq c(x,y).
\end{equation*}
Observe that $\mathcal{K}$ is a convex set that is stable under maxima, contains constants $\mathcal{K}+c\subset\mathcal{K}$ for all $c\in\mathbb{R}$. Moreover, for all $\rho\in\mathcal{K}$ there is
\begin{equation*}
\int_X \rho \,d\mu-\int_Y \rho \,d\nu\leq \int_X \rho_0 \,d\mu-\int_Y \rho_0 \,d\nu.
\end{equation*}
By Corollary \ref{col:twospaces} the extreme points of pairs of Borel probability measures satisfying such inequality are contained in the set of pairs of the form $(\delta_x,\eta)$ with $x\in X$ and $\eta$ a probability measure on $Y$. 
By symmetry, the set of extreme points is contained in the set of pairs of the form $(\delta_{x_0},\delta_{y_0})$for some $x_0\in X$ and $y_0\in Y$. 
For any such extreme point $(\delta_{x_0},\delta_{y_0})$ there is $\rho_0(x_0)-\rho_0(y_0)=c(x_0,y_0)$. Indeed, define $\rho(x)=c(x,y_0)$ for $x\in X$ and set for $y\in Y$
\begin{equation*}
\rho(y)=\sup\{ c(x,y_0)-c(x,y)\mid x\in X\}.
\end{equation*}
Then $\rho\in\mathcal{K}$ and $\rho(x_0)-\rho(y_0)=c(x_0,y_0)$. Thus also $\rho_0(x_0)-\rho_0(y_0)=c(x_0,y_0)$. 

It follows that the considered set $\mathcal{E}$ of extreme points is equal to
\begin{equation*}
\Big\{(\delta_x,\delta_y)\mid \rho_0(x)-\rho_0(y)=c(x,y), x\in X,y\in Y\Big\}.
\end{equation*}
By Choquet's theorem there is a probability measure $\pi_0$ on  $\mathcal{E}$ such that 
\begin{equation}\label{eqn:repr}
(\mu,\nu)=\int_{\mathcal{E}} (\xi_1,\xi_2)\,d\pi_0(\xi).
\end{equation}
Define 
\begin{equation*}
\pi=\int_{\mathcal{E}}\xi_1\otimes \xi_2 \,d\pi_0(\xi).
\end{equation*}
Then, by (\ref{eqn:repr}), $\pi\in\Gamma(\mu,\nu)$ and
\begin{equation*}
\int_{X\times Y}c \,d\pi=\int_X \phi \,d\mu-\int_Y \psi\, d\nu
\end{equation*}
and the proof is complete.
\end{proof}

\begin{remark}\label{rem:semic}
If $c\colon X\times Y\to\mathbb{R}$ is a lower semi-continuous function, then it may be written as a supremum of a sequence of bounded and Lipschitz functions, see e.g. \cite{Villani2}. Applying Theorem \ref{thm:opti} for each function from the sequence, we may obtain the duality result for the function $c$.
\end{remark}

\section{Kantorovich--Rubinstein duality}\label{s:Kant-Rub}

In the present section we present a proof of Kantorovich--Rubinstein duality analogous to the proof in the former section.

\begin{theorem}\label{thm:kr}
Suppose that $\Omega$ is a bounded, locally compact Polish space with metric $d$. Let $\mu$ and $\nu$ be Borel probability measures on $\Omega$. Then the supremum of integrals
\begin{equation*}
\int_{\Omega} g\, d(\mu-\nu)
\end{equation*}
taken over the set of $1$-Lipschitz functions $g\in\mathcal{C}(\Omega)$ is equal to the infimum of integrals
\begin{equation*}
\int_{\Omega \times \Omega} d\,d\pi
\end{equation*}
over all $\pi\in\Gamma(\mu,\nu)$.
Here $\Gamma(\mu,\nu)$ stands for the set of all Borel probability measures on $\Omega\times \Omega$ such that its marginals are $\mu$ and $\nu$ respectively.
Moreover, both supremum and infimum are attained.
\end{theorem}
\begin{proof}
The fact that the supremum is attained follows by Arzel\`a--Ascoli theorem and by Ulam's lemma, cf. Lemma \ref{lem:attains}. Take a $1$-Lipschitz function $f\colon \Omega\to\mathbb{R}$ such that 
\begin{equation}\label{eqn:lipnowe}
\int_{\Omega} g\, d(\mu-\nu)\leq \int_{\Omega} f \,d(\mu-\nu)
\end{equation}
for all $1$-Lipschitz functions $g\in\mathcal{C}(\Omega)$. 
The set $\mathcal{K}$ of all $1$-Lipschitz functions satisfies assumptions of Proposition {pro:extr}. Observe that $\mathcal{K}$ is also stable under minima.
Hence, the set of extreme points of pairs of measures that satisfy (\ref{eqn:lipnowe}) is contained in the set of pairs of the form $(\delta_{x_0},\delta_{y_0})$ for some $x_0,y_0\in\Omega$. We claim that for any such extreme point  there is 
\begin{equation}\label{eqn:lipp}
f(x_0)-f(y_0)=d(x_0,y_0).
\end{equation}
Indeed, fix an extreme point $(\delta_{x_0},\delta_{y_0})$, $x_0,y_0\in\Omega$. Define $f_0(x)=d(x,y_0)$. Then $f_0\in\mathcal{K}$ and $f_0(x_0)-f_0(y_0)=d(x_0,y_0)$. Then (\ref{eqn:lipp}) follows, as
\begin{equation*}
d(x_0,y_0)= f_0(x_0)-f_0(y_0)\leq f(x_0)-f(y_0)\leq d(x_0,y_0).
\end{equation*}
By the Choquet's theorem,  there exists a Borel probability measure $\pi_0$ on the set of extreme points $\mathcal{E}$ such that 
\begin{equation}\label{eqn:ee}
(\mu,\nu)=\int_{\mathcal{E}}\xi\, d\pi_0(\xi).
\end{equation}
Define 
\begin{equation*}
\pi=\int_{\mathcal{E}}\xi_1\otimes \xi_2 \,d\pi_0(\xi).
\end{equation*}
Then $\pi\in\Gamma(\mu,\nu)$, by (\ref{eqn:ee}). Moreover
\begin{equation*}
\int_{\Omega\times \Omega} d(x,y) \,d\pi(x,y)=\int_{\Omega\times \Omega} \big(f(x)-f(y)\big) \,d\pi(x,y)=\int_{\Omega} f\,d(\mu-\nu).
\end{equation*}
\end{proof}

\section{Multi-marginal optimal transport}\label{s:multi}

Here we generalise our approach to the multi-marginal optimal transport with finitely many marginals, see e.g. \cite{Pass} for a previous account on this topic. The duality results have been already established in \cite{Kellerer2}.
In what follows we shall need the following lemma, see \cite{Pass2} for a less general version. The lemma provides a version of the $c$-convexification method employed in multi-marginal optimal transport problems, see e.g. \cite{Carlier}, \cite{Swiech}.

\begin{lemma}\label{lem:convexification}
Let $X_1,\dotsc,X_k$ be metric spaces. Let 
\begin{equation*}
c\colon X_1\times\dotsc\times X_k\to\mathbb{R}
\end{equation*}
be a Lipschitz function. Let $A_i\subset X_i$ for $i=1,\dotsc,k$ and let 
\begin{equation*}
f_i\colon A_i\to\mathbb{R}\text{ for }i=1,\dotsc,k
\end{equation*}
be such that for all $x_i\in A_i$, $i=1,\dotsc,k$, there is
\begin{equation}\label{eqn:inmany}
\sum_{i=1}^k f_i(x_i)\leq c(x_1,\dotsc,x_k).
\end{equation}
Then there exists Lipschitz functions $\tilde{f}_i\colon X_i\to\mathbb{R}$, $i=1,\dotsc, k$, such that condition (\ref{eqn:inmany}) holds true for all $x_i\in X_i$, $i=1,\dotsc, k$. Moreover $f_i(x_i)\leq \tilde{f}_i(x_i)$ for all $x_i\in A_i$ and $i=1,\dotsc, k$.
Each $\tilde{f}_i$, $i=1,\dotsc,k$, may be taken so that its Lipschitz constant is at most the Lipschitz constant of $c$.
\end{lemma}
\begin{proof}
We define inductively $\tilde{f}_i(x_i)$,  $x_i\in X_i$, for $i=1,\dotsc,k$ as
\begin{equation*}
\inf\Big\{c(x_1,\dotsc,x_k)-\sum_{j=1}^{i-1}\tilde{f}_j(x_j)-\sum_{j=i+1}^k f_j(x_j) \mid x_j\in X_j\text{ if }j< i, x_j\in A_j\text{ if }j>i\Big\}.
\end{equation*}
Then $\sum_{i=1}^k\tilde{f}_i(x_i)\leq c(x_1,\dotsc,x_k)$ for $x_i\in X_i$, $i=1,\dotsc,k$, and thus
\begin{equation*}
\tilde{f}_i(x_i)\leq\inf\Big\{c(x_1,\dotsc,x_k)-\sum_{j\neq i}\tilde{f}_j(x_j)\mid x_j \in X_j, j\neq i\Big\}.
\end{equation*}
Moreover $f_i\leq\tilde{f}_i$ on $A_i$ for all $i=1,\dotsc,k$ and thus $\tilde{f}_i$ is at least the infimum on the right-hand side of the above equality. This is to say, for $x_i\in X_i$ and $i=1,\dotsc,k$
\begin{equation*}
\tilde{f}_i(x_i)=\inf\Big\{c(x_1,\dotsc,x_k)-\sum_{j\neq i}\tilde{f}_j(x_j) \mid x_j\in X_j, j\neq i\Big\}.
\end{equation*}
If $c$ was $L$-Lipschitz, then $\tilde{f}_i$, $i=1,\dotsc,k$ are $L$-Lipschitz as infima of $L$-Lipschitz functions.
\end{proof}

\begin{remark}\label{rem:infima}
Pick $x_i\in X_i$, $i=1,\dotsc,k$. Let $f(x_1)=c(x_1,\dotsc,x_k)$ and let $f(x_i)=0$ for $i=2,\dotsc,k$. Then the assumptions of the above lemma are satisfied with $A_i=\{x_i\}$, $i=1,\dotsc,k$. Therefore we may apply the $c$-convexification procedure described above in the proof, to obtain functions $\tilde{f}_i\colon X_i\to\mathbb{R}$, $i=1,\dotsc,k$ such that 
\begin{equation*}
\sum_{i=1}^k \tilde{f}(y_i)\leq c(y_1,\dotsc,y_k)\text{ for all }y_i\in X_i 
\end{equation*}
and moreover 
\begin{equation*}
\sum_{i=1}^k\tilde{f}(x_i)=c(x_1,\dotsc,x_k).
\end{equation*}
\end{remark}
The following lemma is based on \cite[Remark 1.13]{Villani2}.

\begin{lemma}\label{lem:boundedness}
Let $X_1,\dotsc,X_k$ be sets. Let 
\begin{equation*}
c\colon X_1\times\dotsc\times X_k\to\mathbb{R}
\end{equation*}
be a bounded function. Suppose that $f_i\colon X_i\to\mathbb{R}$, $i=1,\dotsc, k$, are such that for all $x_i\in X_i$ and $i=1,\dotsc,k$
\begin{equation*}
f_i(x_i)=\inf\Big\{c(x_1,\dotsc,x_k)-\sum_{j\neq i}f_j(x_j) \mid x_j\in X_j, j\neq i\Big\}.
\end{equation*}
Then there exist constants $h_1,\dotsc,h_k\in\mathbb{R}$ that sum up to zero, such that the functions $\tilde{f}_i=f_i+h_i$ satisfy
\begin{equation*}
\tilde{f}_i(x_i)=\inf\Big\{c(x_1,\dotsc,x_k)-\sum_{j\neq i}\tilde{f}_j(x_j) \mid x_j\in X_j, j\neq i\Big\}
\end{equation*}
and all of them are bounded by the uniform norm of $c$ times $\max\{k, 3\}$.
\end{lemma}
\begin{proof}
Note that for any $h_1,\dotsc,h_k$ that sum up to zero there is 
\begin{equation}\label{eqn:takie}
\inf\Big\{c(x_1,\dotsc,x_k)-\sum_{j\neq i}\tilde{f}_j(x_j) \mid x_j\in X_j, j\neq i\Big\}=f_i(x_i)-\sum_{j\neq i }h_j=\tilde{f}_i(x_i).
\end{equation}
Thus the first assertion is proven. Let $M$ denote the uniform norm of $c$. Choose $h_1,\dotsc,h_k$ in such a way that 
\begin{equation*}
\sup\{\tilde{f}_i(x_i)\mid x_i\in X_i\}=M\text{ for }i=2,\dotsc,k.
\end{equation*}
Note that by (\ref{eqn:takie}) it follows that for $i=1,\dotsc,k$ and all $x_i\in X_i$
\begin{equation*}
-M-\sum_{j\neq i}\sup\{\tilde{f}_j(x_j)\mid x_j\in X_j\}\leq\tilde{f}_i(x_i)\leq M-\sum_{j\neq i}\sup\{\tilde{f}_j(x_j)\mid x_j\in X_j\}.
\end{equation*}
Thus, for all $x_1\in X_1$
\begin{equation}\label{eqn:takiet}
-kM\leq\tilde{f}_1(x_1)\leq (2-k)M.
\end{equation}
Now, again from (\ref{eqn:takie}) and from (\ref{eqn:takiet}), we get that for $i=2,\dotsc,k$ and $x_i\in X_i$
\begin{equation*}
-M-(k-2)M+(k-2)M\leq \tilde{f}_i(x_i)\leq M-(k-2)M+kM.
\end{equation*}
Hence, for such indices $i$,
\begin{equation*}
-M\leq \tilde{f}_i(x_i)\leq 3M.
\end{equation*}
\end{proof}

The following theorem provides a novel interpretation of the Kantorovich problem in the multi-marginal setting as minimisation of a certain linear functional over the set of all Choquet's representations of $k$-tuples of probability measures.

\begin{theorem}\label{thm:optimany}
Let $X_1,\dotsc,X_k$ be locally compact Polish spaces. Let $
c\colon X_1\times\dotsc\times X_k\to\mathbb{R}$ be a bounded Lipschitz function. 
Let $\mu_i$ be a Borel probability measure on $X_i$ for each $i=1,\dotsc,k$.
Then the supremum of sum of integrals
\begin{equation*}
\sum_{i=1}^k\int_{X_i}f_i\, d\mu_i 
\end{equation*}
taken over the set of continuous, bounded functions $f_i\in\mathcal{C}(X_i)$, $i=1,\dotsc,k$, such that
\begin{equation*}
\sum_{i=1}^kf_i(x_i)\leq c(x_1,\dotsc,x_k)\text{ for all }x_i\in X_i, i=1,\dotsc,k
\end{equation*}
is equal to the infimum of integrals
\begin{equation*}
\int_{X_1\times\dotsc\times X_k} c\,d\pi
\end{equation*}
over all $\pi\in \Gamma (\mu_1,\dotsc,\mu_k)$. 
Here $\Gamma (\mu_1,\dotsc,\mu_k)$ stands for the set of all Borel probability measures on $X_1\times\dotsc\times X_k$ such that its marginals on $X_i$ are $\mu_i$ for $i=1,\dotsc,k$. Moreover, both supremum and infimum are attained.
\end{theorem}

\begin{lemma}\label{lem:tuples}
Let $X_1,\dotsc,X_k$ be locally compact Polish spaces. Let $c\colon X_1\times\dotsc\times X_k\to\mathbb{R}$
be a non-negative bounded Lipschitz function,.  
Let $\mathcal{L}$ denote the set of all bounded continuous functions $g\in\mathcal{C}(X_1\cup\dotsc\cup X_k)$ such that for all $x_i\in X_i$, $i=1,\dotsc,k$, we have
\begin{equation*}
\sum_{i=1}^kg(x_i)\leq c(x_1,\dotsc,x_k).
\end{equation*}
Let $f\in \mathcal{L}$.
Then the set of extreme points of the set $\mathcal{P}$ of $k$-tuples of Borel probability measures $(\mu_1,\dotsc,\mu_k)\in\mathcal{P}(X_1)\times\dotsc\times\mathcal{P}(X_k)$ such that 
\begin{equation*}
\sum_{i=1}^k\int_{X_i} g \,d\mu_i\leq \sum_{i=1}^k\int_{X_i} f \,d\mu_i
\end{equation*}
for all $g\in\mathcal{L}$ is equal to the set of $k$-tuples of the form $(\delta_{x_1},\dotsc,\delta_{x_k})$,  such that
\begin{equation*}
\sum_{i=1}^k f(x_i)=c(x_1,\dotsc,x_k).
\end{equation*}
\end{lemma}
\begin{proof}
For any  $l\in\{1,\dotsc,k\}$ let $I_l=\{1,\dotsc,l-1,l+1,\dotsc,k\}$ and let $\Omega$ be the disjoint union of all $X_i$, $i=1,\dotsc,k$. Let $(\mu_1,\dotsc,\mu_k)\in\mathcal{P}$. We shall denote by $\tilde{\mu}_l$  the extension of $\mu_l$ to $\Omega$. Let $\mu_{I_l}$ denote the probability measure on $\Omega$ given by 
\begin{equation*}
\mu_{I_l}=\frac1{k-1}\sum_{i\neq l}\tilde{\mu}_i.
\end{equation*}
Then, for any $g\in \mathcal{L}$, we have
\begin{equation*}
\int_{\Omega}g\,d\mu_l+\int_{\Omega}(k-1)g\,d\mu_{I_l}\leq \int_{\Omega}f\,d\mu_l+\int_{\Omega}(k-1)f\,d\mu_{I_l}.
\end{equation*}
Denote by $X_{I_l}$ the disjoint union of $X_i$, $i\in I_l$. Let $\mathcal{L}_l$ denote the convex set of all continuous bounded functions on $\Omega$ which are equal to $g$ on $X_l$ and to $-(k-1)g$ on $X_{I_l}$ for some $g\in\mathcal{L}$.
Then, $\mathcal{L}_l$ is stable under maxima, contains constants for any constant $t$ there is $t+\mathcal{L}_l\subset\mathcal{L}_l$. Moreover for any $h\in\mathcal{L}_l$ there is
\begin{equation*}
\int_{X_l}h\,d\mu_l-\int_{X_{I_l}}h\,d\mu_{I_l}\leq \int_{X_l}f\,d\mu_l-\int_{X_{I_l}}(1-k)f\,d\mu_{I_l}.
\end{equation*}

By Corollary \ref{col:twospaces}, the extreme points of the set $\mathcal{P}_l$ of pairs of Borel probability measures $(\mu,\nu)\in \mathcal{P}(X_l)\times \mathcal{P}\big(\bigcup_{i\in I_l}X_i\big)$ such that 
\begin{equation*}
\int_{X_l}h\,d\mu-\int_{X_{I_l}}h\,d\nu\leq \int_{X_l}f\,d\mu-\int_{X_{I_l}}(1-k)f\,d\nu
\end{equation*}
for all $h\in\mathcal{L}_l$ 
are of the form $(\delta_x,\eta)$ for some probability $\eta\in \mathcal{P}\big(\bigcup_{i\in I_l}X_i\big)$. 
By the Choquet's theorem there exists a Borel probability measure $\pi_l$ on the set $\mathcal{E}_l$ of extreme points of $\mathcal{P}_l$ such that
\begin{equation*}
(\mu_l,\mu_{I_l})=\int_{\mathcal{E}_l}\xi \,d\pi_l(\xi).
\end{equation*}
Hence for any $i\in I_l$
\begin{equation*}
\mu_i=\int_{\mathcal{E}_l}(k-1)\xi_2|_{X_i}\,d\pi_l(\xi).
\end{equation*}
Here we write $\xi=(\xi_1,\xi_2)$ for $\xi\in \mathcal{E}_l$.
It follows that $\pi_l$-almost all $(k-1)\xi_2|_{X_i}$ are probabilities. We may write
\begin{equation*}
(\mu_1,\dotsc,\mu_l,\dotsc,\mu_k)=\int_{\mathcal{E}_l} \big((k-1)\xi_2|_{X_1},\dotsc,\xi_1,\dotsc,(k-1)\xi_2|_{X_k}\big)\,d\pi_l(\xi).
\end{equation*}       
Observe that for $\pi_l$-almost every $\xi$ there is     
\begin{equation*}
\big((k-1)\xi_2|_{X_1},\dotsc,\xi_1,\dotsc,(k-1)\xi_2|_{X_k}\big)\in\mathcal{P}.
\end{equation*}     
Hence, any extreme point of $\mathcal{P}$ has to be of the form 
\begin{equation*}
(\eta_1,\dotsc,\eta_{l-1},\delta_{x_l},\eta|_{X_{l+1}},\dotsc,\eta|_{X_k})
\end{equation*}
with $x_l\in X_l$ and some probability measures $\eta_i$ for $i\in I_l$.
As this holds true for any $l=1,\dotsc,k$, any extreme point of $\mathcal{P}$ has to have the form $(\delta_{x_1},\dotsc,\delta_{x_k})$ with $x_i\in X_i$, $i=1,\dotsc,k$.

Take now any extreme point $(\delta_{x_1},\dotsc,\delta_{x_k})$ of $\mathcal{P}$ and let $f\in\mathcal{L}$ be as in the statement of the lemma. Then for any $g\in\mathcal{L}$ we have
\begin{equation}\label{eqn:cosik}
\sum_{i=1}^kg(x_i)\leq\sum_{i=1}^k f(x_i).
\end{equation}
By Remark \ref{rem:infima} there exists a function $g\in\mathcal{L}$ such that
\begin{equation*}
\sum_{i=1}^k g(x_i)=c(x_1,\dotsc,x_k).
\end{equation*}
By (\ref{eqn:cosik}) it follows that also
\begin{equation*}
\sum_{i=1}^k f(x_i)=c(x_1,\dotsc,x_k).
\end{equation*}
The proof is complete.
\end{proof}

\begin{proof}[of Theorem \ref{thm:optimany}]
The fact that the supremum is attained follows by Lemmata \ref{lem:convexification}, \ref{lem:boundedness}, by Ulam's lemma and by Arzel\`a--Ascoli theorem, cf. Lemma \ref{lem:attains}.

The assertion follows from Lemma \ref{lem:tuples} and by the Choquet's theorem, cf. Theorem \ref{thm:opti}. Indeed, if $\pi_0$ is a Borel probability measure on the set of extreme points $\mathcal{E}$ of $\mathcal{P}$ of the previous lemma then an optimal $\pi\in\Gamma(\mu_1,\dotsc,\mu_k)$ is given by the formula
\begin{equation*}
\pi=\int_{\mathcal{E}}\xi_1\otimes \dotsc\otimes \xi_k \,d\pi_0(\xi),
\end{equation*}
where $\xi=(\xi_1,\dotsc,\xi_k)\in \mathcal{E}$.
\end{proof}

\begin{remark}
If $c\colon X_1\times \dotsc\times X_k\to\mathbb{R}$ is a lower semi-continuous function, then it may be written as a supremum of a sequence of bounded and Lipschitz functions, see e.g. \cite{Villani2}. Applying Theorem \ref{thm:optimany} for each function from the sequence, we may obtain duality result for the function $c$.
\end{remark}


  \bibliographystyle{plain} 
  \bibliography{refs}

\begin{thebibliography}{10}

\bibitem{Alfsen}
E.~M. Alfsen.
\newblock {\em Compact convex sets and boundary integrals}.
\newblock Ergebnisse der Mathematik und ihrer Grenzgebiete. Springer-Verlag,
  1971.

\bibitem{Aliprantis}
C.~D. Aliprantis and K.~C. Border.
\newblock {\em Infinite Dimensional Analysis: a Hitchhiker's Guide}.
\newblock Springer, Berlin; London, 2006.

\bibitem{Asplund}
E.~Asplund.
\newblock {Fr{\'e}chet differentiability of convex functions}.
\newblock {\em Acta Mathematica}, 121(none):31 -- 47, 1968.

\bibitem{Asplund2}
E.~Asplund and R.~Rockafellar.
\newblock Gradients of convex functions.
\newblock {\em Transactions of the American Mathematical Society},
  139:443--467, 1969.

\bibitem{Aze}
D.~Az\`e and J.-P. Penot.
\newblock Uniformly convex and uniformly smooth convex functions.
\newblock {\em Annales de la Facult\'e des sciences de Toulouse:
  Math\'ematiques}, Ser. 6, 4(4):705--730, 1995.

\bibitem{Becker}
R.~Becker.
\newblock {\em Convex Cones in Analysis}.
\newblock Collection Travaux en cours : math{\'e}matiques. Hermann, 2006.

\bibitem{Beiglbock5}
M.~Beiglb{\"o}ck, L.~Christian, and W.~Schachermayer.
\newblock A general duality theorem for the {M}onge--{K}antorovich transport
  problem.
\newblock {\em Studia Mathematica}, 209:151--167, 2012.

\bibitem{Beiglbock}
M.~Beiglb{\"o}ck, A.~M.~G. Cox, and M.~Huesmann.
\newblock Optimal transport and {S}korokhod embedding.
\newblock {\em Inventiones mathematicae}, 208(2):327--400, May 2017.

\bibitem{Beiglbock1}
M.~Beiglb\"{o}ck and N.~Juillet.
\newblock On a problem of optimal transport under marginal martingale
  constraints.
\newblock {\em Ann. Probab.}, 44(1):42--106, 2016.

\bibitem{Beiglbock2}
M.~Beiglb{\"o}ck, T.~Lim, and J.~Ob{\l}{\'o}j.
\newblock Dual attainment for the martingale transport problem.
\newblock {\em Bernoulli}, 25(3):1640--1658, 08 2019.

\bibitem{Beiglbock6}
M.~Beiglb{\"o}ck and W.~Schachermayer.
\newblock Duaulity for {B}orel measurable cost functions.
\newblock {\em Transactions of the American Mathematical Society},
  363(8):4203--4224, 2011.

\bibitem{Bowles}
M.~Bowles.
\newblock {\em Linear transfers, {K}antorovich operators, and their ergodic
  properties}.
\newblock PhD thesis, University of British Columbia, 2020.

\bibitem{Caravenna2}
L.~Caravenna and S.~Daneri.
\newblock The disintegration of the {L}ebesgue measure on the faces of a convex
  function.
\newblock {\em Journal of Functional Analysis}, 258(11):3604 -- 3661, 2010.

\bibitem{Carlier}
G.~Carlier and B.~Nazaret.
\newblock Optimal transportation for the determinant.
\newblock {\em ESAIM: Control, Optimisation and Calculus of Variations},
  14(4):678--698, 2008.

\bibitem{Cavalletti}
F~Cavalletti and A.~Mondino.
\newblock Sharp and rigid isoperimetric inequalities in metric-measure spaces
  with lower {R}icci curvature bounds.
\newblock {\em Inventiones mathematicae}, 208(3):803--849, 2017.

\bibitem{Choquet}
G.~Choquet, J.E. Marsden, T.~Lance, and S.S. Gelbart.
\newblock {\em Lectures on Analysis}.
\newblock Mathematics lecture note series. W. A. Benjamin, 1969.

\bibitem{Ciosmak2}
K.J. Ciosmak.
\newblock Leaves decompositions in {E}uclidean spaces.
\newblock {\em Journal de Math{\'e}matiques Pures et Appliqu{\'e}es},
  154:212--244, 2021.

\bibitem{Ciosmak3}
K.J. Ciosmak.
\newblock Optimal transport of vector measures.
\newblock {\em Calculus of Variations and Partial Differential Equations},
  (60:230), 2021.

\bibitem{deAcosta}
A.~De~Acosta.
\newblock Invariance principles in probability for triangular arrays of
  b-valued random vectors and some applications.
\newblock {\em The Annals of Probability}, pages 346--373, 1982.

\bibitem{DeMarch3}
H.~{De March}.
\newblock {Local structure of multi-dimensional martingale optimal transport}.
\newblock {\em arXiv e-prints}, page arXiv:1805.09469, May 2018.

\bibitem{DeMarch2}
H.~{De March}.
\newblock {Quasi-sure duality for multi-dimensional martingale optimal
  transport}.
\newblock {\em arXiv e-prints}, page arXiv:1805.01757, May 2018.

\bibitem{DeMarch1}
H.~{De March} and N.~{Touzi}.
\newblock {Irreducible convex paving for decomposition of multi-dimensional
  martingale transport plans}.
\newblock {\em arXiv e-prints}, page arXiv:1702.08298, Feb 2017.

\bibitem{Dudley1}
R.~M. Dudley.
\newblock {\em Probabilities and metrics: Convergence of laws on metric spaces,
  with a view to statistical testing}, volume~45.
\newblock University of Aarhus, 1976.

\bibitem{Dudley2}
R.~M. Dudley.
\newblock {\em Distances of Probability Measures and Random Variables}, pages
  28--37.
\newblock Springer New York, New York, NY, 2010.

\bibitem{Dunn}
J.C. Dunn.
\newblock Convexity, monotonicity, and gradient processes in {H}ilbert space.
\newblock {\em Journal of Mathematical Analysis and Applications},
  53(1):145--158, 1976.

\bibitem{Fernique}
X.~Fernique.
\newblock Sur le theoreme de {K}antorovitch-{R}ubinstein dans les espaces
  polonais.
\newblock In J.~Az{\'e}ma and M.~Yor, editors, {\em S{\'e}minaire de
  Probabilit{\'e}s XV 1979/80}, pages 6--10, Berlin, Heidelberg, 1981. Springer
  Berlin Heidelberg.

\bibitem{Fuchssteiner}
B.~Fuchssteiner.
\newblock An abstract disintegration theorem.
\newblock {\em Pacific Journal of Mathematics}, 94(2):303--309, 1981.

\bibitem{Fuchssteiner1}
B.~{Fuchssteiner} and W.~{Lusky}.
\newblock {\em {Convex cones}}, volume~56.
\newblock Elsevier, Amsterdam, 1981.

\bibitem{Galichon}
A.~Galichon, P.~Henry-Labord\`ere, and N.~Touzi.
\newblock A stochastic control approach to no-arbitrage bounds given marginals,
  with an application to lookback options.
\newblock {\em Ann. Appl. Probab.}, 24(1):312--336, 2014.

\bibitem{Swiech}
W.~Gangbo and A.~\'Swiech.
\newblock Optimal maps for the multidimensional {M}onge-{K}antorovich problem.
\newblock {\em Communications on Pure and Applied Mathematics}, 51(1):23--45,
  1998.

\bibitem{Ghoussoub2}
N.~Ghoussoub.
\newblock Linear transfers as minimal costs of dilations of measures in
  balayage order.
\newblock {\em arXiv e-prints}, page arXiv:2212.05152v1, 2022.

\bibitem{Ghoussoub}
N.~Ghoussoub, Y.-H. Kim, and T.~Lim.
\newblock Structure of optimal martingale transport plans in general
  dimensions.
\newblock {\em Ann. Probab.}, 47(1):109--164, 01 2019.

\bibitem{Hackenbroch}
W.~Hackenbroch.
\newblock A non-commutative {S}trassen disintegration theorem.
\newblock In D.~K{\"o}lzow, editor, {\em Measure Theory Oberwolfach 1979},
  pages 424--430, Berlin, Heidelberg, 1980. Springer Berlin Heidelberg.

\bibitem{Kantorovich2}
L.V. Kantorovich.
\newblock {On a problem of {M}onge}.
\newblock {\em {J. Math. Sci., New York}}, 133(4):1383, 2006.

\bibitem{Kantorovich}
L.V. Kantorovich.
\newblock On the translocation of masses.
\newblock {\em Journal of Mathematical Sciences}, 133(4):1381--1382, Mar 2006.

\bibitem{Kellerer2}
H.~G. Kellerer.
\newblock Duality theorems for marginal problems.
\newblock {\em Zeitschrift f{\"u}r Wahrscheinlichkeitstheorie und Verwandte
  Gebiete}, 67(4):399--432, Nov 1984.

\bibitem{Pass}
Y.~Kim and B.~Pass.
\newblock A general condition for {M}onge solutions in the multi-marginal
  optimal transport problem.
\newblock {\em SIAM Journal on Mathematical Analysis}, 46(2):1538--1550, 2014.

\bibitem{Klartag}
B.~{Klartag}.
\newblock {Needle decompositions in Riemannian geometry}.
\newblock {\em Memoirs of the American Mathematical Society}, 249(1180), Jun
  2017.

\bibitem{Lucchetti}
R.~Lucchetti and F.~Patrone.
\newblock Hadamard and {T}yhonov well-posedness of a certain class of convex
  functions.
\newblock {\em Journal of Mathematical Analysis and Applications},
  88(1):204--215, 1982.

\bibitem{McShane}
E.~J. McShane.
\newblock Extension of range of functions.
\newblock {\em Bull. Amer. Math. Soc.}, 40(12):837--842, 1934.

\bibitem{Meyer}
P.A. Meyer.
\newblock {\em Probability and potentials}.
\newblock Blaisdell book in pure and applied mathematics. Blaisdell Pub. Co.,
  1966.

\bibitem{Meyer-Nieberg}
P.~Meyer-Nieberg.
\newblock Strassen disintegration theorems.
\newblock {\em Archiv der Mathematik}, 65:310--315, 1995.

\bibitem{Mikami}
T.~Mikami.
\newblock {A simple proof of duality theorem for {M}onge-K{}antorovich
  problem}.
\newblock {\em Kodai Mathematical Journal}, 29(1):1 -- 4, 2006.

\bibitem{Mikami2}
T.~Mikami and M.~Thieullen.
\newblock Duality theorem for the stochastic optimal control problem.
\newblock {\em Stochastic Processes and their Applications},
  116(12):1815--1835, 2006.

\bibitem{Mokobodzki1}
D.~Mokobodzki, G.~Sibony.
\newblock C\^ones de fonctions et th\'eorie du potentiel i. {L}es noyaux
  associ\'es \`a un c\^one de fonctions.
\newblock {\em Séminaire Brelot-Choquet-Deny. Théorie du potentiel},
  11:1--35, 1966-1967.

\bibitem{Mokobodzki2}
G.~Mokobodzki and D.~Sibony.
\newblock C\^ones de fonctions et th\'eorie du potentiel ii. {R}\'esolvantes et
  semi-groupes subordonn\'es \`a un c\^one de fonctions.
\newblock {\em S\'eminaire Brelot-Choquet-Deny. Th\'eorie du potentiel}, 11,
  1966-1967.
\newblock talk:9.

\bibitem{Monge}
G.~Monge.
\newblock M{\'e}moire sur la th{\'e}orie des d{\'e}blais et des remblais.
\newblock In {\em {H}istoire de l’{A}cad{\'e}mie {R}oyale de {S}ciences de
  {P}aris}, pages 666--704. 1781.

\bibitem{Neumann}
M.~Neumann.
\newblock On the {S}trassen disintegration theorem.
\newblock {\em Archiv der Mathematik}, 29(1):413--420, Dec 1977.

\bibitem{Obloj}
J.~Ob\l\'{o}j.
\newblock The {S}korokhod embedding problem and its offspring.
\newblock {\em Probab. Surv.}, 1:321--390, 2004.

\bibitem{Siorpaes}
J.~{Ob{\l}{\'o}j} and P.~{Siorpaes}.
\newblock {Structure of martingale transports in finite dimensions}.
\newblock {\em arXiv e-prints}, page arXiv:1702.08433, Feb 2017.

\bibitem{Pass2}
B.~Pass.
\newblock Uniqueness and {M}onge solutions in the multimarginal optimal
  transportation problem.
\newblock {\em SIAM J. Math. Analysis}, 43:2758--2775, 2011.

\bibitem{Pass1}
B.~Pass.
\newblock Multi-marginal optimal transport: theory and applications.
\newblock {\em ESAIM: M2AN}, 49(6):1771--1790, 2015.

\bibitem{Phelps}
R.R. Phelps.
\newblock {\em Lectures on {C}hoquet's {T}heorem}.
\newblock Lecture Notes in Mathematics. Springer Berlin Heidelberg, 2003.

\bibitem{Ramachandran}
D.~Ramachandran and L.~R{\"u}schendorf.
\newblock A general duality theorem for marginal problems.
\newblock {\em Probability Theory and Related Fields}, 101(3):311--319, 1995.

\bibitem{Ramachandran2}
D.~Ramachandran and L.~R{\"u}schendorf.
\newblock Duality and perfect probability spaces.
\newblock {\em Proceedings of the American Mathematical Society},
  124(7):2223--2228, 1996.

\bibitem{Rockafellar3}
R.~Rockafellar.
\newblock Monotone operators and the proximal point algorithm.
\newblock {\em SIAM Journal on Control and Optimization}, 14:877--898, 1976.

\bibitem{Smulyan}
V.~L. {\v{S}}mulyan.
\newblock Sur la de\'erivabilit\'e de la norme dans l'espace de {B}anach.
\newblock {\em Dokl. Acad. Naukl. SSSR}, 27:643--648, 1940.

\bibitem{Strassen}
V.~Strassen.
\newblock The existence of probability measures with given marginals.
\newblock {\em Ann. Math. Statist.}, 36(2):423--439, 04 1965.

\bibitem{Valadier}
M.~Valadier.
\newblock On the {S}trassen theorem.
\newblock In Jean-Pierre Aubin, editor, {\em Analyse Convexe et Ses
  Applications}, pages 203--215, Berlin, Heidelberg, 1974. Springer Berlin
  Heidelberg.

\bibitem{Villani2}
C.~Villani.
\newblock {\em Topics in optimal transportation}, volume~58 of {\em Graduate
  Studies in Mathematics}.
\newblock American Mathematical Society, Providence, RI, 2003.

\bibitem{Villani1}
C.~Villani.
\newblock {\em Optimal transport}, volume 338 of {\em Grundlehren der
  Mathematischen Wissenschaften [Fundamental Principles of Mathematical
  Sciences]}.
\newblock Springer-Verlag, Berlin, 2009.
\newblock Old and new.

\bibitem{Wayne}
Mathematics Department Coffee~Room Wayne State~University.
\newblock Every convex function is locally {L}ipschitz.
\newblock {\em The American Mathematical Monthly}, 79(10):1121--1124, 1972.

\bibitem{Winkler}
G.~Winkler.
\newblock Extreme points of moment sets.
\newblock {\em Mathematics of Operations Research}, 13(4):581--587, 1988.

\bibitem{Zalinescu}
C.~Z\u{a}linescu.
\newblock On uniformly convex functions.
\newblock {\em J. Math. Anal. Appl.}, 95(2):344--374, 1983.

\end{thebibliography}
  
\affiliationone{
   Krzysztof J. Ciosmak\\
Fields Institute for Research in Mathematical Sciences, 222 College Street, Toronto, Ontario M5T 3J1, Canada\\
Department of Mathematics, University of Toronto, Bahen Centre, 40 St. George St., Room 6290, Toronto, Ontario, M5S 2E4, Canada
   \email{k.ciosmak@utoronto.ca}}

%


\end{document}